\documentclass[12pt,oneside,reqno]{amsart}

\usepackage[a4paper,left=25mm,top=32mm,right=25mm,bottom=32mm]{geometry}

\usepackage{amssymb,amsfonts}
\usepackage[all,arc]{xy}
\usepackage{enumerate}
\usepackage{siunitx}
\usepackage{enumitem}
\usepackage{mathtools}
\usepackage{mathrsfs}
\usepackage{dsfont}
\usepackage{amsfonts}
\usepackage{amssymb}
\usepackage{graphicx}
\usepackage{ragged2e}
\usepackage{amsthm}
\usepackage{amsmath}
\usepackage[mathscr]{euscript}
 \let\mathscr\relax
\usepackage[scr]{rsfso}
\usepackage{graphicx}
\usepackage{color}
\usepackage{float}
\usepackage{caption}

\DeclareMathOperator{\vol}{Vol}

\DeclareMathOperator{\Spec}{Spec}

\newtheorem{thm}{Theorem}[section]
\newtheorem{cor}[thm]{Corollary}
\newtheorem{prop}[thm]{Proposition}
\newtheorem{lem}[thm]{Lemma}

\newtheorem{claim}[thm]{Claim}

\theoremstyle{definition}
\newtheorem{defn}[thm]{Definition}

\theoremstyle{remark}

\numberwithin{equation}{section}

\author{Andrea Sartori}
\address[Andrea Sartori]{Department of Mathematics, Tel Aviv University, Tel Aviv, Israel, IL}
\email[Andrea Sartori]{Sartori.Andrea.math@gmail.com}
\begin{document}

		\title[Nodal length of toral eigenfunctions]{Asymptotic nodal length and log-integrability of toral eigenfunctions}
		\begin{abstract}
We study the nodal set of Laplace eigenfunctions on the flat $2d$ torus $\mathbb{T}^2$. We prove an asymptotic law for the nodal length of such eigenfunctions, under some growth assumptions on their Fourier coefficients. Moreover, we show that their nodal set is asymptotically equidistributed on $\mathbb{T}^2$.	The proofs are based on Bourgain's de-randomisation technique and	the main new ingredient, which might be of independent interest, is the integrability of arbitrarily large powers of the doubling index of Laplace eigenfunctions on $\mathbb{T}^2$, based on the work of Nazarov \cite{N93,Nun}.
		\end{abstract}

		\maketitle
	\section{Introduction}
		\subsection{Nodal length of Laplace eigenfunctions and the Random Wave Model}
	Given a compact $C^{\infty}$-smooth  Riemannian surface $(M,g)$ without boundary, let $\Delta_g$ be the associated Laplace-Beltrami operator. We are interested in the eigenvalue problem
		\begin{align}
	\nonumber	\Delta_g f_{\lambda}+ \lambda f_{\lambda}=0. 
		\end{align}
Since $M$ is compact, the spectrum of $-\Delta_g$ is a discrete subset of $\mathbb{R}$ with only accumulation point at $+\infty$. The eigenfunctions $f_{\lambda}$ are smooth and their nodal set, that is their zero set, is a smooth $1d$ sub-manifold outside a finite set of points \cite{C76}. In particular, the Hausdorff measure of the nodal set is well-defined and called the  \textit{nodal length}  
\begin{align}
	\nonumber \mathcal{L}(f_\lambda):= \mathcal{H}\{x\in M: f_\lambda(x)=0\}.
	\end{align}

	 Yau \cite{Ysurvey}, and independently Br\"{u}ning \cite{B78}, showed that $\mathcal{L}(f_{\lambda})\geq c\lambda^{1/2}$ for some  $c=c(M)>0$. Yau \cite{Ysurvey}   conjectured the matching upper bound 
	 	\begin{align} 
	 	\nonumber	c\sqrt{\lambda}\leq  \mathcal{L}(f_\lambda)\leq C \sqrt{\lambda},
	 		\end{align}
	 	for some $C=C(M)>0$. Donnelly and Fefferman \cite{DF} showed that Yau's conjecture holds for manifolds  of any dimension, provided that the metric is  real-analytic. Recently, Logunov \cite{L1,L2} and  Logunov-Malinnikova \cite{LM} proved the optimal lower-bound for $C^{\infty}$ manifolds and gave a polynomial upper-bound.

\vspace{2mm}	

	Some heuristic insight into the behavior of the nodal length can be   deduced from a conjecture of  Berry \cite{B1,B2}, known as the Random Wave Model (RWM). The RWM asserts that, on a generic chaotic surface, Laplace eigenfunctions restricted to balls of radius $\approx\lambda^{-1/2}$, the so-called Planck scale,  should behave like the isotropic Gaussian field $F$ with covariance function
	$$\mathbb{E}[F(x)\overline{F(y)}]= J_0\left( |x-y|\right),$$
	where $J_0(\cdot)$ is the $0$-th Bessel function. Berry \cite{B1} found the expected nodal length of $F$ on a box $B$ of unit side length to be 
	$$	\mathbb{E}[\mathcal{L}(F,B)]:=\mathbb{E}[\mathcal{H}\{x\in B: F(x)=0\}]= \frac{1}{2\sqrt{2}}.$$
 Covering $M$ by balls/boxes of Planck-scale radius,  the RWM suggests not only the global behavior
	\begin{align}
		\mathcal{L}(f_{\lambda})= \frac{\vol(M)\sqrt{\lambda}}{2\sqrt{2}}(1+o_{\lambda\rightarrow \infty}(1)), \label{1}
		\end{align}
 but also the macroscopic distribution
	\begin{align}
		\mathcal{L}(f_{\lambda},B)= \frac{ \vol (B) \sqrt{\lambda}}{2\sqrt{2}} (1+o_{\lambda\rightarrow \infty}(1)), \label{2}
	\end{align}
for any ball $B=B(r)$ of fixed, that is independent of $\lambda$, radius $r>0$. In particular, we expect the nodal set to be asymptotically equidistributed on $M$, see also \cite[Chapter 13]{Zbook}. 

\vspace{2mm}

We study a class of \textit{deterministic} Laplace eigenfunctions  on the standard two dimensional torus $\mathbb{T}^2=\mathbb{R}^2/\mathbb{Z}^2$ with moderate growth of their Fourier coefficients. These are known as \textquotedblleft flat\textquotedblright eigenfunctions, see section \ref{notation} below. The main result is that, within the said class, the asymptotic law \eqref{2}, up to a possibly different leading constant, holds in every ball of fixed radius, along a density one sub-sequence of eigenvalues\footnote{Let $S\subset \mathbb{R}$ be some (infinite) sequence, a subsequence $S'\subset S$ has density one if $\lim\limits_{X\rightarrow \infty} |\{\lambda \in S': \lambda \leq X\}|/|\{\lambda \in S: \lambda \leq X\}|=1$. }. While the behavior of the nodal length of \textit{random} Laplace eigenfunctions  has been intensively studied \cite{BMW20,KKW,MRW16,RW08,W10}, to the best of the author knowledge, no other, non-random or non-trivial (e.g. $f_{\lambda}(x)= \cos(a \cdot x)$ with $|a|^2=\lambda$), examples of \eqref{2} or even \eqref{1} are known. Thus the results of this manuscript seem to be the first to address the asymptotic behavior of the nodal length of deterministic Laplace eigenfunctions.

\vspace{2mm}

The proof of the main result is based on the de-randomisation technique pioneered by Bourgain \cite{BU} and developed by Buckley-Wigman \cite{BW16}.  Bourgain's de-randomization asserts that flat eigenfunctions behave according to the RWM in most balls of Planck-scale radius, see Proposition \ref{prop: convergence in distribution} below. In order to apply this technique to study the nodal set, it is thus essential to control the zero set in the balls failing the RWM-type behavior.  In light of Donnelly-Fefferman work \cite{DF}, it is well-understood that, in the real-analytic setting, the nodal set in a ball $B$ can be controlled by the doubling index $N(B)$, a measure of the growth of the function (see section \ref{doubling index sec} below). This leads us to the study, of possible independent interest, of the distribution of the doubling index at Planck-scale, for flat eigenfunctions:   Given any $q>1$, we show that
 \begin{align}
 	\label{anti-concentration}
 	\int_{\mathbb{T}^2} N_{f_{\lambda}}(B(x,\lambda^{-1/2}))^q  dx\leq C,
 \end{align}
 for some $C=C(q)>0$. This requires a combination of some Fourier-analytic techniques borrowed from the work of Nazarov \cite{N93,Nun}, and some arithmetic considerations.  We point out that Yau's conjecture is equivalent to \eqref{anti-concentration} with $q=1$ \cite{R15}. In this direction, our work seems to be the first to address the higher-integrability properties of the doubling index.

\subsection{Statement of the main results}
\label{notation}
Before stating our main results we need to introduce some notation pertaining to Laplace eigenfunctions on $\mathbb{T}^2$. The eigenvalues of $-\Delta$  are, up to a factor of $4\pi^2$,  integers representable as the sum of two squares $\lambda\in 	S:= \{\lambda\in \mathbb{Z}: \lambda= \square + \square\}$ and have multiplicity $N=N(\lambda):=|\{\xi\in \mathbb{Z}^2: |\xi|^2=\lambda\}|$ given by the number of lattice points on the circle of radius $\lambda^{1/2}$. Any toral eigenfunction, with eigenvalue $-4\pi^{2} \lambda$ (we will simply say eigenvalue $\lambda$ from now on), can be expressed as a Fourier sum 
\begin{align}
	f_\lambda(x)=f(x)= \sum_{\substack{\xi\in \mathbb{Z} \\ |\xi|^2=\lambda}}a_{\xi}e(\xi \cdot x ), \label{function}
\end{align}
where $e(\cdot)= \exp(2\pi i \cdot )$ and the $a_{\xi}$'s are complex numbers satisfying $\overline{a_{\xi}}=a_{-\xi}$ for every $\xi$, making $f_{\lambda}$ real valued. Moreover, we normalize $f_{\lambda}$ so that
\begin{align}
\label{normalisation}
||f_{\lambda}||^2_{L^2(\mathbb{T}^2)}= \sum|a_{\xi}|^2=1.
\end{align}

\vspace{2mm}

We first consider the special class of  \textit{Bourgain}'s eigenfunctions, that is functions as in \eqref{function} whose Fourier coefficients satisfy 
\begin{align}
	\label{Bourgain}
	|a_{\xi}|^2= N^{-1},
\end{align}
for all $|\xi|^2=\lambda$. Bourgain's eigenfunctions are especially important in that they precisely behave as predicted by the RWM, that is they resemble, locally almost everywhere, the centered Gaussian random field with covariance $J_0(\cdot)$. In particular, the asymptotic law for their nodal length can be stated  directly without the need for extra notation: 
 \begin{thm}
	\label{thm 1.1}
 There exists a density one subsequence $S'\subset S$  such that for $\lambda\in S'$ the following holds: let  $B\subset \mathbb{T}^2$ be a fixed ball or $B=\mathbb{T}^2$, then we have   
	$$ \mathcal{L}(f_{\lambda},B)= \frac{ \vol (B)}{2\sqrt{2}} (4\pi^2\lambda)^{1/2} (1+o_{\lambda \rightarrow \infty}(1)),$$
	uniformly for all $f_{\lambda}$ as in \eqref{function} satisfying \eqref{Bourgain}. 
\end{thm}

	We point out that the sequence $S'\subset S$ postulated in Theorem \ref{thm 1.1} (and Theorem \ref{thm 1} below) can be described explicitly via some conditions of pure arithmetic nature, see section \ref{sec: arithmetic background} below. We also stress that the rate of convergence in Theorem \ref{thm 1.1} (and Theorem \ref{thm 1} below) does depend on $B$. However, it is \textit{plausible} that the techniques developed in this manuscript, combined with some recent work on lattice points \cite{KS21}, could be pushed forward to show that Theorem \ref{thm 1.1} (and Theorem \ref{thm 1} below) holds in any ball $B$ of radius larger than the Planck-scale, $r>\lambda^{-1/2+\varepsilon}$. This would imply an essentially optimal equidistribution regime for the nodal length. We leave this question to be addressed elsewhere. 

	\vspace{2mm}

We will now introduce the class of flat toral eigenfunctions and some additional notation required to describe their nodal length. \begin{defn}
		\label{def flatness}
	 Fix some positive function $u:\mathbb{R}\rightarrow \mathbb{R}_{>0}$ such that, for every $\varepsilon>0$, $u(N)=o_{N\rightarrow \infty}(N^{\varepsilon})$. A function $f_{\lambda}$ as in \eqref{function} is said to be \textit{flat} if 
	\begin{align}
	\nonumber
		\sup_{|\xi|^2=\lambda}|a_{\xi}|^2\leq \frac{u(N)}{N}.
	\end{align}
\end{defn}
Even though the definition of flat eigenfunctions depends on the particular choice of the function $u$, this will only affect the rate of convergence in Theorem \ref{thm 1} below. Therefore, in order not to overburden the notation, we fix $u$ throughout the whole manuscript. 

\vspace{2mm}
 
 As we will see, flat eigenfunction, as Bourgain's eigenfunction, also behave, locally almost everywhere, as a Gaussian field. However, the covariance structure of the said field, and eventually its nodal length,  depend  on the  measure 
\begin{align}
	\label{spectral measure1}
&	\mu_f=\sum_{\xi} |a_{\xi}|^2 \delta_{\xi/\sqrt{\lambda}} 
\end{align}
where $\delta_{\xi/\sqrt{\lambda}}$ is the Dirac distribution at the point $\xi/\sqrt{\lambda}$, and its Fourier coefficients
$$\widehat{\mu_f}(k)= \int_{\mathbb{S}^1} z^k d\mu_f(z),$$
  $k \in \mathbb{Z}$ and $\mathbb{S}^1\subset \mathbb{R}^2$ is the unit circle.    

\vspace{2mm}

In order to simplify the exposition of the main result, it will be useful to arrange (sequences) of functions $f_{\lambda}$ according to the possible  weak$^{\star}$ limits of $\mu_f$, see \cite{KW17,S18} for a study of the said weak$^{\star}$ limits. First,  observe that $\mu_f$ is a probability measure with support contained in the unit circle $\mathbb{S}^{1}\subset \mathbb{R}^2$ and that the set of probability measures on $\mathbb{S}^1$, equipped with the weak$^{\star}$ topology, is compact. Thus, upon passing to a subsequence, we may (and will) assume that 
\begin{align}
\label{spectral measure}	&\mu_f\longrightarrow \mu &N\rightarrow \infty,
	\end{align} 
where the convergence is with respect to the weak$^{\star}$ topology, for some symmetric\footnote{$\mu(-A)=\mu(A)$ for any measurable set $A\subset \mathbb{S}^1$.} probability measure  $\mu$ on $\mathbb{S}^1$.    Moreover, to avoid degeneracies, we \textit{assume} that the support of $\mu$ is not contained in a line. Sorting (sequences of) functions $f_{\lambda}$ according to their limiting measure avoids an unnecessary dependence (on $f_{\lambda}$) of the leading constant in the following result:

 \begin{thm}
 	\label{thm 1}
 There exists a density one subsequence $S'\subset S$ such that the following holds. Let $\{f_{\lambda}\}_{\lambda\in S'}$ be a sequence of flat, in the sense of Definition \ref{def flatness}, eigenfunctions  with limiting measure  $\mu$ in the sense of \eqref{spectral measure}. Then, for any fixed ball $B\subset \mathbb{T}^2$ or $B=\mathbb{T}^2$, we have  
 $$ \mathcal{L}(f_{\lambda},B)= c_1 \vol (B) (4\pi^2 \lambda)^{1/2} (1+o_{\lambda \rightarrow \infty}(1)),$$
   where
   $$	c_1= \frac{1-|\widehat{\mu}(2)|^2}{2^{5/2}\pi}\int_{0}^{2\pi}\frac{1}{(1-\alpha\cos(2\theta)-\beta \sin(2\theta))^{3/2}}d\theta,$$
  and $\widehat{\mu}(2)= \alpha+i\beta$. 
 \end{thm}

The dependence of the nodal length of toral eigenfunction on the measure $\mu$, as in \eqref{spectral measure}, was already observed, in the \textit{random} setting, by Kurlberg, Krishnapur and Wigman \cite{KKW}. They found that the variance of the nodal length depends on the fourth, as opposed to the second, Fourier coefficient of $\mu$, while the expectation is universal. On one hand, Theorem \ref{thm 1} shows that the nodal length behavior is much richer than what can be captured by random models. And, on the other hand, it precisely describes how the distribution of lattice points affects the nodal length. 

\vspace{2mm}

The main new ingredient, instrumental to the proof of \eqref{2}, which will allow us to show  \eqref{anti-concentration}, is the log-integrability of $f$:
\begin{prop}
	\label{log-integrability}
	Let $q\geq 1$ be an integer. Then there exists a density one subsequence of $S'=S'(q)\subset S$ and some constant $C=C(q)>0$ such that for all $\lambda\in S'$ the following holds: for every flat $f_{\lambda}$, in the sense of Definition \ref{def flatness}, we have 
	$$\int_{\mathbb{T}^2} |\log |f_{\lambda}(x)||^{q} dx \leq C.$$
\end{prop} 
The flatness assumption is not essential for the proof of Proposition \ref{log-integrability} and it can be removed at the cost of a slightly lengthier calculation in section \ref{sec: prelimiaries to log-integrability}. For the sake of keeping the exposition as simple as possible and since flatness is essential to theorems \ref{thm 1.1} and \ref{thm 1}, we decided to present the proof of Proposition \ref{log-integrability} under the flatness assumption. 
\subsection{Notation}
\label{sec: notation}
To simplify the exposition we adopt the following standard notation: we write $A\lesssim B$ and $A\gtrsim B$ to designate the existence of an absolute constant $C>0$ such that $A\leq C B$ and $A\geq C B$. The letters $C,c$ will be used to designate positive constants which may change from line to line. Moreover, for some parameter $\beta>0$, we write $A=O_{\beta}(B)$ to mean that there exists some constant $C=C(\beta)>0$ such that $|A|\leq C B$, if no parameter is specified in the notation, then the constant is absolute. We write $o_{\beta\rightarrow \infty}(1)$ for any function that tends to zero as $\beta\rightarrow \infty$. Finally, given some function $ g: \mathbb{T}^2 \rightarrow \mathbb{R}$ and a parameter $t>0$, we will use the following shorthand notation: 
$\vol(x\in \mathbb{T}^2: g(x)\leq t) =: \vol(g(x)\leq t)$. 

\section{Preliminaries}

\subsection{Convergence of random fields}
\label{proability background}
The proof of Theorem \ref{thm 1.1} and Theorem \ref{thm 1} is based on studying the restriction of $f_{\lambda}=f$ as in \eqref{function} to the box $B(x,1/\sqrt{\lambda})$
\begin{align}
	\label{F_x}
 F_x(y)= f\left( x+\frac{y}{\sqrt{\lambda}}\right),
\end{align}
where $y\in [-1/2,1/2]^2$, on average as $x$ ranges uniformly over a fixed ball $B\subset \mathbb{T}^2$ (or $B=\mathbb{T}^2$).  We observe that $F_x$ can also be thought as a random field from the \textquotedblleft probability\textquotedblright space $(B,d\vol_B)$, with  $d\vol_B= d\vol/\vol(B)$, into $C^{\infty}([-1/2,1/2]^2)$, the space of infinitely differentiable functions on the unit square $[-1/2,1/2]^2$. In order to distinguish these two points of view and to keep track of the dependence on $B$, we write $F_x^B$ for the random field and $F_x$ for the restriction of $f$ around the point $x\in \mathbb{T}^2$.

 Bourgain's de-randomization asserts that  $F_x^B$, converges in distribution, in the appropriate space of functions, to $F_{\mu}$, the Gaussian field with spectral measure $\mu$ given by \eqref{spectral measure}. In this section, we gather the relevant probabilistic background to rigorously express this claim.  We start by  briefly collecting some definitions and notation about Gaussian fields (on $\mathbb{R}^2$).

\vspace{2mm}

\textbf{Gaussian fields}. Let  $\Omega$ be an abstract probability space, with probability measure $\mathbb{P}(\cdot)$ and expectation $\mathbb{E}[\cdot]$. A (real-valued) Gaussian field $F$ is a continuous map $F: \mathbb{R}^2 \times \Omega\rightarrow \mathbb{R}$ such that all  finite dimensional distributions $(F(x_1, \cdot),...F(x_n,\cdot))$ are multivariate Gaussian vectors. We say that $F$ is \textit{centered} if $\mathbb{E}[F]\equiv0$ and \textit{stationary} if its law is invariant under translations $x\rightarrow x+\tau$ for $\tau \in \mathbb{R}^2$. The \textit{covariance} function of $F$ is 
\begin{align}
	\mathbb{E}[F(x)\cdot F(y)]= \mathbb{E}[F(x-y)\cdot F(0)]. \nonumber
\end{align}
Since the covariance is positive definite, by Bochner's theorem, it is the Fourier transform of some measure $\mu$ on $\mathbb{R}^2$. So we have 
\begin{align}
	\mathbb{E}[F(x)F(y)]= \int_{\mathbb{R}^2} e\left(\langle x-y, s \rangle\right)d\mu(s). \nonumber
\end{align}
The measure $\mu$ is called the \textit{spectral measure} of $F$. Since $F$ is real-valued, $\mu$ is symmetric, that is  $\mu(-A)=\mu(A)$ for any (measurable) subset $A\subset \mathbb{R}^2$. By Kolmogorov's theorem, $\mu$ fully determines $F$. Thus, from now on, we will simply write $F=F_{\mu}$ for the centered, stationary Gaussian field with spectral measure $\mu$. Next, we will describe the metric for the aforementioned convergence of random fields.  
\vspace{2mm}

\textbf{The L\'{e}vy–Prokhorov metric.}  Let $C^s(V)$ be the space of $s$-times, $s\ge0$ integer, continuously differentiable functions on $V$, a compact subset of $\mathbb{R}^2$. Since $C^s(V)$ is a separable metric space, Prokhorov's Theorem,  see \cite[Chapters 5 and 6]{BI}, implies that $\mathcal{P}(C^s(V))$, the space of probability measures on $C^s(V)$, is metrizable via the \textit{L\'evy–Prokhorov metric}. This is defined as follows:  for a (measurable) subset $A\subset C^s(V)$, denote by $A_{{+\varepsilon }}$ the  $\varepsilon$-neighborhood of $A$, that is 
$$
A_{{+\varepsilon }}:=\{p\in C^s(V)~|~\exists~ q\in A,\ ||p-q||<\varepsilon \}=\bigcup _{{p\in A}}B(p,\varepsilon),
$$
where $||\cdot||$ is the $C^{s}$-norm and $B(p,\varepsilon)$ is the (open) ball centered at $p$ of radius $\varepsilon>0$. The \textit{L\'evy–Prokhorov metric} $d_P :{\mathcal  {P}(C^s(V))\times\mathcal{P}}(C^s(V))\to [0,+\infty )$ is defined for two probability measures $\mu$  and $\nu$  as:
\begin{align}\nonumber
	d_P (\mu ,\nu ):=\inf \left\{\varepsilon>0: \mu (A)\leq \nu (A_{{+\varepsilon }})+\varepsilon, \ \nu (A)\leq \mu (A_{{+\varepsilon }})+\varepsilon \ \forall~ A\subset C^{s}(V)\right\}. 
\end{align}

\textbf{Convergence of random functions.}
\label{sec: as a random fields}   We are now ready to describe the metric for the convergence of $F_x^B$ to $F_{\mu}$, with $\mu$ as in \eqref{spectral measure}. Given an integer $s\geq 1$, $F_x^B$ induces a probability measure on $C^s([-1/2,1/2]^2)$ via the push-forward measure 
$$ (F_{x})_{\star}\vol_B(A)=\vol_B(\{x\in B: F_x(\cdot)\in A\}) ,$$
where $A\subset C^s([-1/2,1/2]^2)$ is a measurable subset. Similarly, the push-forward of $F_{\mu}$ defines a probability measure on  $C^s([-1/2,1/2]^2)$ which we denote by  $(F_{\mu})_{\star}\mathbb{P}$. We can now measure the distance between $F_x^B$ and $F_{\mu}$ as the distance between their push-forward measures in  $\mathcal{P}(C^s([-1/2,1/2]^2))$, the space of probability measures on $C^s([-1/2,1/2]^2)$,  equipped with the L\'evy–Prokhorov metric. Therefore,  to shorten notation, we will write 
$$ d_P(F^B_{x},F_{\mu}):= d_P((F_{x})_{\star}\vol_B,(F_{\mu})_{\star}\mathbb{P}).$$
\subsection{Arithmetic background}
\label{sec: arithmetic background}
In order to study the zero set of $f_{\lambda}$ in \eqref{function}, we will need some control over its level sets, $\{x\in \mathbb{T}^2: |f(x)|< t \}$ for $t\in (0,\infty)$. In section \ref{sec: prelimiaries to log-integrability}, this will be accomplished by intersecting the said level sets with horizontal and vertical lines. Thus, we will need some information about the restriction of $f_{\lambda}$ to  horizontal and vertical lines. These are function on $L^2(\mathbb{T})$ with spectrum consisting of the projections of the $\xi$'s, as in \eqref{function}, onto the first and second coordinate.  We collect here some facts about the additive structure of these spectra. 

\vspace{2mm}

Given $\lambda\in S$, and some positive integer $\ell>0$, let $\xi^1,...,\xi^{\ell}$ be $\ell$ points on the circle $|\xi|^2=\lambda$. We are interested in the number of solutions to the linear equation 
\begin{align} 
&	\label{semi-correlations equation} \xi_{i}^1+...+\xi_{i}^{\ell}=0, &i=1,2, \end{align}
 where $\xi^j= (\xi_1^j,\xi_2^j)$. Solutions to \eqref{semi-correlations equation} are called \textit{semi-correlations} and have been first studied in \cite{CKW20}, generalizing an argument of Bombieri and Bourgain \cite{BB15}.  Let $S_{\ell}$ be the set of permutations on $\ell$-tuples, when $\ell=2k$ is even, the set of $\ell$-tuples
 $$ \mathcal{T}_{i}(\lambda,\ell)=\{\pi (\xi_i^1, -\xi_i^1,...,\xi_i^{k},-\xi_i^{k}): \pi \in S_{\ell}\}$$
 is the set of \textit{trivial} solutions to \eqref{semi-correlations equation}, that is the set of $\ell$-tuples canceling out in pairs. We call any other solution to \eqref{semi-correlations equation} \textit{non-trivial}. In particular, when $\ell$ is odd, we say that there are no trivial solutions, $\mathcal{T}_{i}(\lambda,\ell)=\emptyset$ . 
 
 \vspace{2mm}
 
 For a density one subsequence of $S$, the number of solutions to \eqref{semi-correlations equation} has been computed precisely in \cite[Theorem 1.3]{CKW20}. Although \cite[Theorem 1.3]{CKW20} is stated in a form weaker than what we need, the proof gives verbatim the following:
\begin{lem}
	\label{semi-correlations}
	Let $\xi=(\xi_1,\xi_2)\in \mathbb{Z}^2$ and $\ell>0$ be an integer. Then, for a density one subsequence of $\lambda\in S$, there exists no non-trivial solution to the linear equation 
	\begin{align}
		& \xi_{i}^1+...+\xi_{i}^{\ell}=0& |\xi|^2=\lambda \hspace{5mm} i=1,2. \nonumber
	\end{align}
That is, all solutions have the form  $\xi_i^1=-\xi_i^2$, ... ,  $\xi^{\ell-1}_i=-\xi^{\ell}_i$, up to permutations. In particular, there are no solutions when $\ell$ is odd. 
\end{lem}

It will also be important that the number of lattice points on the circle $|\xi|^2=\lambda$ tends to infinity as $\lambda\rightarrow \infty$, this is not always the case as circles with prime (congruent to $1$ modulo $4$) radius have only 8 lattice points. However, the following consequence of the Erd\"{o}s-Kac Theorem, see for example \cite[Part III Chapter 3]{Tan} and \cite[Lemma 2.3]{Simrn}, assures that there are always sufficiently many lattice points.  
\begin{lem}
	\label{ERKA}
	There exists a density one subsequence $S'\subset S$ such that for all sufficiently large $\lambda \in S'$ we have 
	$$N\geq (\log \lambda)^{1/8}.$$ 
	\end{lem} 
In particular, Lemma \ref{ERKA} ensures that, up to the rate of convergence, the limits $\lambda\rightarrow \infty$ and $N\rightarrow \infty$ are equivalent. To shorten the exposition, throughout the manuscript, we assume that \textit{every} (density one) subsequence $S'\subset S$ satisfies the conclusion of Lemma \ref{ERKA}. 

\section{Bourgain's de-randomisation: Asymptotic behavior of the nodal length}

Before embarking in the proof of theorems \ref{thm 1.1} and \ref{thm 1}, we will establish some notation and conventions that we will use through the rest of the manuscript. First, we observe that, even if $F_x$ as in \eqref{F_x} is defined on $B(1)$, the box of side $1$ centered at the origin, we can assume that $F_x$ is well defined on $B(R)$, for any \textit{fixed} parameter $R>1$. This observation will be useful because we will often need to slightly change the scale at which we study $F_x$ (from the unit box to the box of side (say) 20). 

 Let $B\subset \mathbb{T}^2$ be a ball, $\mu$ be some symmetric probability measure on unit circle $\mathbb{S}^1$ and $F_x^B$ be as in section \ref{proability background}. We write 
$$ \mathcal{L}(F_x^B) := \vol(\{y\in [-1/2,1/2]^2: F^B_x(y)=0\}),$$
and 
$$\mathcal{L}(F_{\mu}):=\vol(\{y\in [-1/2,1/2]^2: F_{\mu}(y)=0\}).$$
Note that the function $\mathcal{L}(\cdot)$ \textit{always} denotes the nodal length in the unit box and $\mathcal{L}(F_x^B)$ is a random variable on $(B,d\vol_B)$.  Finally,  in order to shorten some statements, when we say that a function $f_{\lambda}$ as in \eqref{function} is flat, from now on, we always mean in the sense of Definition \ref{def flatness}.

The aim of this section is to prove that $\mathcal{L}(F_x^B)$, as a random variable on $(B,d\vol_B)$, converges in distribution, in the sense of \cite[Theorem 2.1]{BI}, to $\mathcal{L}(F_{\mu})$. In other words, we will prove that $\mathcal{L}(F_x^B)$ is close to $\mathcal{L}(F_{\mu})$ outside a small set of \textquotedblleft bad\textquotedblright  $x\in \mathbb{T}^2$.  Formally, we have the following: 
\begin{prop}
	\label{prop: convergence in distribution}
	Let, $\varepsilon>0$ and $F^B_x$ be as in section \ref{sec: as a random fields}.  There exists a density one subsequence $S'=S'(\varepsilon)\subset S$, such that the following holds: let $\{f_{\lambda}\}_{\lambda\in S'}$ be a sequence of flat eigenfunctions  with limiting measure  $\mu$ in the sense of \eqref{spectral measure} then
	\begin{align}
		\nonumber &\mathcal{L}(F^B_x)\overset{d}{\longrightarrow} \mathcal{L}(F_{\mu})&\lambda\rightarrow \infty,
	\end{align}	
	where the convergence is in distribution, uniformly for all balls $B\subset \mathbb{T}^2$ of radius $r>\lambda^{-1/2+\varepsilon}$. 
\end{prop}
We will use the conclusion  of Proposition \ref{prop: convergence in distribution} only for fixed balls $B\subset\ \mathbb{T}^2$. However, as the proof of the stronger claim does not require any additional argument, we decided to include it in the manuscript. The first step in the proof of Proposition \ref{prop: convergence in distribution} consists of showing that $F_x^B$ converges, in the sense of section \ref{sec: as a random fields}, to $F_{\mu}$. This fact has been shown at macroscopic scales in \cite{BU,BW16} and at microscopic scales in \cite[Proposition 4.5]{S20}:
\begin{lem}
	\label{lem: de-randomisation}
	Let $R\geq 1$,  $\varepsilon>0$ and $F^B_x$ be as in section \ref{sec: as a random fields}.  There exists a density one subsequence $S'=S'(\varepsilon)\subset S$ such that the following holds: let $\{f_{\lambda}\}_{\lambda\in S'}$ be a sequence of flat eigenfunctions  with limiting measure  $\mu$ in the sense of \eqref{spectral measure} then, recalling the notation in section \ref{sec: as a random fields}, we have
	\begin{align}
		\nonumber &d_P(F^{B}_x,F_{\mu})=d_P((F_{x})_{\star}\vol_B,(F_{\mu})_{\star}\mathbb{P})\rightarrow 0 &\lambda\rightarrow \infty,
	\end{align}
in the space $\mathcal{P}(C^2([-R/2,R/2]^2))$, where the convergence is uniform for all balls $B\subset \mathbb{T}^2$ of radius $r>\lambda^{-1/2+\varepsilon}$, but depends on $R$. 
\end{lem}
Since we use a different formulation from \cite{BU,BW16,S20}, we will briefly justify Lemma \ref{lem: de-randomisation}:
\begin{proof}[Proof of Lemma \ref{lem: de-randomisation}]
	Let $\Omega$ be the abstract probability space where $F_{\mu}$ is defined and let $\delta>0$ be given. Under the assumptions of Lemma \ref{lem: de-randomisation}, 	\cite[Lemma 4.4]{S20} and \cite[Proposition 4.5]{S20}   states that,  for all sufficiently large $\lambda \in S'$,  there exists a map $\tau: \Omega\rightarrow B$ and a subset $\Omega'\subset \Omega$, both independent of $f_{\lambda}$, such that: 
	\begin{enumerate}
		\item For any measurable $A\subset \Omega$, $\vol(\tau(A))= \pi r^2 \mathbb{P}(A)$,
		\item $\mathbb{P}(\Omega')\leq \delta$,
		\item For all $\omega\not\in \Omega'$, 
		$$||F_{\mu}(\tau(\omega),y)- F_x^B(y)||_{C^2[-R/2,R/2]^2}\leq ||F_{\mu}(\tau(\omega),Ry)- F_x^B(Ry)||_{C^2[-1/2,1/2]^2}\leq R^2\delta.$$
	\end{enumerate} 
	Therefore, given a measurable set $A\subset C^2([-1/2,1/2]^2)$, we have 
	\begin{align}
		(F_\mu)_{\star}\mathbb{P}(A)&= \mathbb{P}(	F_\mu(\omega) \in A)= \mathbb{P}(F_\mu(\omega)\in A, \omega \in \Omega') + \mathbb{P}(F_\mu(\omega)\in A, \omega \not\in \Omega') \nonumber \\
		&  \stackrel{(2)}{\leq} \mathbb{P}(F_\mu(\omega)\in A, \omega \not\in \Omega') +\delta  \stackrel{(1)-(3)}{\leq}   (F_{x})_{\star}\vol_B (A_\delta) +R^2\delta +\delta \nonumber
	\end{align}
	Similarly, we have 
	$ (F_{x})_{\star}\vol_B (A) \leq 	(F_\mu)_{\star}\mathbb{P}(A_{\delta}) +2R^2\delta.$
	Hence, since $\delta$ is arbitrary and $R$ is fixed, we obtain $d_P(F^{B}_x,F_{\mu}) \rightarrow 0,$ as required. 
\end{proof}
 The second step in the proof of Proposition \ref{prop: convergence in distribution} consists of showing that we can pass from the convergence of $F_x$ to $F_{\mu}$ to the convergence of their nodal sets. The following lemma shows that the nodal length is a continuous functional on the appropriate (open) subspace of  $C^2$, see \cite{NS} and \cite[Lemma 6.1]{Simrn}. The precise form of this fact, as stated below, can be found in  \cite[Lemma 6.1]{RS20}:
\begin{lem}
	\label{continuity L}Let $B\subset \mathbb{R}^2$ be a ball/box, let $2B$ be the concentric ball/box of twice the radius/side and let  $C^2_{*}(2B):=\{g\in C^2(2B):  |g|+|\nabla g|>0 \}$. Then $\mathcal{L}(g, B)=\vol( \{x\in B: g(x)=0\})$ is a continuous functional on $C^2_{*}(2B)$. 
\end{lem}
In light of Lemma \ref{continuity L}, Proposition \ref{prop: convergence in distribution} would follow from Lemma \ref{lem: de-randomisation} via the Continuous Mapping Theorem, provided  that $F_{\mu}\in C^2_{*}$. This is a well-known result of Bulinskaya, see  \cite[Lemma 6]{NS}.
\begin{lem}[Bulinskaya's lemma]
	\label{Bulinskya's Lemma}
	Let $F=F_{\mu}$, with $\mu$ a symmetric measure supported on $\mathbb{S}^1$ and $B(2)\subset \mathbb{R}^2$ be the box of side $2$ centered at zero. If $\mu$ is not supported on a line, that is $(F,\nabla F)$ is non-degenerate, then $F\in C^2_*(B(2))$ almost surely, with $C^2_*(B(2))$ as in Lemma \ref{continuity L}. 
\end{lem}

We are finally ready to prove Proposition \ref{prop: convergence in distribution}
\begin{proof}[Proof of Proposition \ref{prop: convergence in distribution}]
Let $S'\subset S$ be given by Lemma \ref{lem: de-randomisation} with (say) $\varepsilon=1/4$. First, applying Lemma \ref{lem: de-randomisation} with (say) $R=4$, we obtain
$$d_P(F^{B}_x,F_{\mu}) \rightarrow 0,$$
	with respect to the $C^2(B(2))$ topology. Moreover,	since the support of $\mu$ is not contained in a line, Lemma \ref{Bulinskya's Lemma} implies that $F_{\mu}\in C^2_*(B(2))$ almost surely. Hence, Lemma \ref{continuity L} together with  the Continuous Mapping Theorem \cite[Theorem 2.7]{BI} imply Proposition \ref{prop: convergence in distribution}, as required.
\end{proof}

	\section{Log-integrability and level-sets estimates}
We formulate (a slightly stronger version of) Proposition \ref{log-integrability} in terms of level sets estimates as follows: 
\begin{prop}
	\label{prop: level sets estimates} 
	Let $q\geq 1$ be an integer.  There exist a density one subsequence of $S'=S'(q)\subset S$, $\lambda_0=\lambda_0(q)>0$ and $\alpha=\alpha(q)>0$ such that the following holds: uniformly for all flat $f_{\lambda}$ in \eqref{function}, with $\lambda>\lambda_0$ in $S'$, and all $t\in (0,\infty)$, we have 
	$$\vol\left(x\in \mathbb{T}^2: \log|f_{\lambda}(x)|<-t^{1/q}\right)\lesssim_q t^{-1+\alpha}.$$
\end{prop}
We are now going to prove Proposition \ref{log-integrability} assuming Proposition \ref{prop: level sets estimates}.
\begin{proof}[Proof of Proposition \ref{log-integrability} assuming Proposition \ref{prop: level sets estimates}]
	As mentioned in section \ref{sec: notation}, given $ g: \mathbb{T}^2 \rightarrow \mathbb{R}$ and a parameter $t>0$, we will use the  shorthand notation $$\vol(x\in \mathbb{T}^2: g(x)\leq t) =: \vol(g(x)\leq t).$$
	We are no ready to begin the proof of  Proposition \ref{log-integrability}.	Let $q\geq 1$ be given, write $f=f_{\lambda}$ and let $S'$ be given by Proposition \ref{prop: level sets estimates}. First, we observe that, by a straightforward integration by parts\footnote{Note that, by Cauchy-Schwarz and \eqref{normalisation}, we have  $\sup_x |f(x)|\leq \sqrt{N}$ so $f$ cannot assume arbitrarily large values.}, we have 
\begin{align}
	\label{eq 1.1}
	\int_{\mathbb{T}^2} |\log|f(x)||^qdx= \int_0^{\infty} t \hspace{1mm}  d\vol( |\log|f(x)||^q\leq t) 	= - \int_0^{\infty} t \hspace{1mm}  d\vol( |\log|f(x)||^q\geq t) \nonumber \\
	= \int_{10}^{\infty} \left(\vol \left( \log|f(x)|\geq t^{1/q}\right) +  \vol \left( \log|f(x)|\leq  -t^{1/q}\right)\right) dt +O(1). 
\end{align}
Since $||f||_{L^2}=1$, Chebyshev's inequality gives 
$$ \vol \left( |f(x)|\geq \exp(t^{1/q})\right) \leq \exp(-2t^{1/q}),$$
thus the first term on the r.h.s. of \eqref{eq 1.1} is bounded by some constant depending on $q$ only. Proposition \ref{prop: level sets estimates} implies that the second term on the r.h.s. of \eqref{eq 1.1} is also bounded   by some constant depending on $q$ only, for all sufficiently large $\lambda \in S'$. By discarding at most finitely many elements of $S'$, we may assume that the claimed bound holds for all $\lambda\in S'$. This concludes the proof of Proposition \ref{log-integrability}. 
\end{proof}
The rest of the section is dedicated to the proof of Proposition \ref{prop: level sets estimates}.
\subsection{Nazarov's result: $\Lambda(p)$-systems and level-sets estimates}
The aim of this section is to present (some of) the results of \cite[Chapter 3]{N93} and \cite{Nun} in a form that it will be useful to prove Proposition \ref{prop: level sets estimates}, we claim no originality and refer the reader to  see directly \cite{N93,Nun}. 

We need to first introduce some definitions: given some $g\in L^2(\mathbb{T})$, the \textit{spectrum} of $g$ is 
$$ \Spec(g):= \left\{n \in \mathbb{Z}:  \widehat{g}(n):= \int_{\mathbb{T}} e(n \cdot x)g(x)dx \neq 0\right\}.$$
We say that a (possibly finite) set $V=\{n_i\}_i\subset \mathbb{Z}$ is a  $\Lambda(p)$-system for some $p\geq 2$ if, for every $g\in L^2(\mathbb{T})$ with $\Spec(g)\subset V$, there exists some constant $C_0=C_0(V,p)>0$, independent of $g$, such that 
\begin{align}
	||g||_{L^p(\mathbb{T})}\leq C_0 ||g||_{L^2(\mathbb{T})}. \label{def tildeC}
\end{align}
We say that a set $V\subset \mathbb{Z}$ is symmetric if $n\in V$ implies $-n\in V$. We will need the following sufficient condition for a symmetric set to be a $\Lambda(p)$-system:
\begin{claim}
	\label{no non-trivial sol}
	Let $V=\{n_i\}_i\subset \mathbb{Z}$ be a symmetric set. Suppose that, for some even $p\geq 2$, the only solutions to
	$$ n_{i_1}+ n_{i_2}+...+ n_{i_p}=0$$
	are trivial, that is $n_{i_1}=-n_{i_2}$ ... , up to permutations. Then, $V$ is a $\Lambda(p)$-system with constant $C_0(p)=c(p)$ independent of $V$. 
\end{claim}

\begin{proof}
	Let $g\in L^2(\mathbb{T})$ with $\Spec(g)\subset V$, we may write $g$ as  
	\begin{align} \nonumber 
		g(x)= \sum_i a_i e(n_i \cdot x),
	\end{align}
	for some $a_i \in \mathbb{C}$. Normalizing $g$, we may assume that	$ ||g||_{L^2}= \sum_i |a_i|^2=1.$
	Now, expanding the $p$-th power of $g$, we have
	$$ ||g||_{L^p}^p= \sum_{i_1,...,i_p} a_{i_1}\overline{a_{i_2}}a_{i_3}... \overline{a_{i_p}} \int_{\mathbb{T}} e(\langle n_{i_1}- n_{i_2}+...- n_{i_p}, x \rangle) dx.$$
	Using  the orthogonality of the exponentials and the assumptions of Claim \ref{no non-trivial sol} (note that, since $V$ is symmetric, the choice of signs in the sum is irrelevant) we deduce
	$$||g||_{L^p}^p= \sum_{\substack{i_1,...,i_p \\ n_{i_1}- n_{i_2}+...- n_{i_p}=0}} a_{i_1}\overline{a_{i_2}}a_{i_3}... \overline{a_{i_p}}= c(p) \left(\sum_i |a_{i}|^2\right)^{p/2}= c(p)$$ 
	where $c(p)$ is the number of permutations of $n_{i_1}=-n_{i_2}$.... and therefore  independent of $g$ and of $V$, as required.
\end{proof}
The last piece of notation that we need is the following: given some $V=\{n_i\}_i\subset \mathbb{Z}$, we denote
\begin{align} 
	&	R(V):= \sup_{\substack{r\in \mathbb{Z} \\ r\neq 0}} |\{(n_i,n_j)\in V^2: n_i-n_j=r\}|, &D(V):=\{n_i-n_j \in \mathbb{Z}: i\neq j\}. \label{definition} 
\end{align}
With the above notation, we have the following theorem from \cite{Nun} whose proof will be given, for completeness, in Appendix \ref{sec: proof of Nazarov's Theorem}. 
\begin{thm}[Nazarov]
	\label{Thm: Logintegrability}
	Let $\varepsilon>0$, $V\subset \mathbb{Z}$ and $R(V)$, $D(V)$ be as in \eqref{definition}. Suppose that $R(V)<\infty$ and $D(V)$ is a $\Lambda(p)$-system for some integer $p>2$ with $C_0=C_0(V,p)$ as in \eqref{def tildeC}. Then there exists some constant $C=C(C_0,\varepsilon, R(V))>0$ such that, uniformly for all  $g\in L^2(\mathbb{T})$ with spectrum contained in $V$ and any set $U\subset \mathbb{T}$ of positive measure, we have  
	$$||g||^2_{L^2(\mathbb{T})}\leq \exp \left( \frac{C}{\rho(U)^{\frac{4}{p}+\varepsilon}}\right)\int_U |g(x)|^2dx, $$
	 where $\rho(\cdot)$ is the (normalized) Lebesgue measure on $\mathbb{T}$.
\end{thm}
We will need the following corollary: 
\begin{cor}
	\label{cor: Nazarov}
	Under the assumptions of Theorem \ref{Thm: Logintegrability} and maintaining the same notation, there exists some constant $C=(C_0,\varepsilon,R(V)) >0$ such that, uniformly for all $g\in L^2(\mathbb{T})$ with spectrum contained in V and satisfying $||g||^2_{L^2}\geq 1/2$, we have  
	 $$\rho \left(x \in \mathbb{T} : \log |g(x)|\leq -t^{\frac{4}{p}+ \frac{\varepsilon}{2}}\right) \leq C t^{-1- \varepsilon/3}.$$
\end{cor}
The constant $1/2$ in the postulated lower bound for $||g||^2_{L^2}$ in the statement of Corollary \ref{cor: Nazarov} is arbitrary and it could be substituted by any other (absolute) constant.  
\begin{proof}
	 Let $0<\delta<1$ and define the set $U_{\delta}= \{x\in \mathbb{T}: |g(x)|\leq \delta\}$. Theorem \ref{Thm: Logintegrability}, applied to $U=U_{\delta}$ with some $\varepsilon_1>0$ to be chosen later, implies that there exists some $C=C(C_0,\varepsilon_1, R(V))>0$ such that
	$$1 \leq  2\exp \left( \frac{C}{\rho(U_{\delta})^{\frac{4}{p}+\varepsilon_1}}\right) \rho(U_{\delta}) \delta ^2 \leq  \exp \left( \frac{100C}{\rho(U_{\delta})^{\frac{4}{p}+\varepsilon_1}}\right) \delta ^2.$$
	 Therefore,  for some $C_1=C_1(C,\varepsilon_1)>0$ and some $c=c(p)>0$, we have 
	\begin{align}
		\rho (U_{\delta})\leq C_1 (-\log \delta)^{-\frac{p}{4}+c\varepsilon_1}. \nonumber
	\end{align}
	Taking $\delta= \exp(-t^{\frac{4}{p}+ \frac{\varepsilon}{2}})$ and choosing $\varepsilon_1$ appropriately in terms of $\varepsilon$ and $p$, we deduce 
	$$  \rho \left( x\in \mathbb{T}: \log |g(x)|\leq -t^{\frac{4}{p}+ \frac{\varepsilon}{2}}\right) \leq C_1 t^{-1- \varepsilon/3},$$
	as required.
\end{proof}
\subsection{Preliminaries for the proof of Proposition \ref{prop: level sets estimates}}
\label{sec: prelimiaries to log-integrability}
Before embarking in the proof of Proposition \ref{prop: level sets estimates}, we will need to set up some relevant notation and make a couple of observations, which, for convenience, we collect in this section.  Given $t\in (0,\infty)$ and an integer $q\geq 1$,  we will bound the volume of the set
\begin{align}
\label{def A_{t,q}}	A=A_{t,q}:=\{x\in \mathbb{T}^2: \log|f_{\lambda}(x)|<-t^{1/q} \}
\end{align}
by using Corollary \ref{cor: Nazarov} to estimate the Lebesgue measure of its intersection with horizontal lines.  To this end, we need to introduce some notation. First, for $\xi \in \mathbb{Z}^2$ and $x\in \mathbb{T}^2$,  we write  $\xi=(\xi_1,\xi_2)$ and $x=(x_1,x_2)$. Second, for $f_{\lambda}=f$ as in \eqref{function} and some fixed $x_2\in \mathbb{T}$, we write
\begin{align}
\label{def; f^1}
 H_{x_2}(f)(\cdot)= f(\cdot, x_2):=  \sum_{\xi_1} b_{\xi_1} e(\xi_1\cdot),
\end{align}
where
  $$b_{\xi_1}= b(\xi,x_2):= a_{(\xi_1,\xi_2)}e(\xi_2\cdot x_2) + a_{(\xi_1,-\xi_2)} e(-\xi_2\cdot x_2). \nonumber$$ 
   That is, $H_{x_2}f$ is $f$ considered as a function of the $1$-th coordinate only and therefore  $H_{x_2}f\in L^2(\mathbb{T})$. In particular, the $L^2$-norm of $H_{x_2}f$ is a function of $x_2$ only. We denote (the square of) this function by $P$: 
\begin{align} 
\label{eq: L^2-norm} P(x_2)&= P_{f}(x_2):=	||H_{x_2}f||^2_{L^2}= \int_{\mathbb{T}} |f(x_1,x_2)|^2 dx_1 =\sum_{\xi_1} |b_{\xi_1}|^2 \\
&= \sum_{\xi_1} \left(|a_{(\xi_1,\xi_2)}|^2+|a_{(\xi_1,-\xi_2)}|^2\right) + Q(x_2)= 1+ Q(x_2),\nonumber
	\end{align}
where 
\begin{align}
&	Q(x_2)= Q_{f}(x_2):= \sum_{\xi_2} d_{\xi_2}e(2\xi_2\cdot x_2)  \nonumber
& d_{\xi_2}= d(\xi):=  a_{(\xi_1,\xi_2)} a_{(-\xi_1,\xi_2)}, 
\end{align}
and in \eqref{eq: L^2-norm} we have used the normalization $\sum_{\xi} |a_{\xi}|^2=1$ and the fact that $\overline{a_{\xi}}=a_{-\xi}$.  
Suppose that $f$ is flat as in Definition \ref{def flatness}, then the $L^2$- norm of $P$ is
$$	||P||^2_{L^2}= 1 + \int_{\mathbb{T}} Q(x_2)dx_2 + 	\int_{\mathbb{T}} \overline{ Q}(x_2)dx_2 + ||Q||^2_{L^2}= 1 + ||Q||^2_{L^2} + O(u(N)\cdot N^{-1/2}), $$
where the error term comes from (possible) terms with $\xi_2=0$ and the flatness assumption in Definition \ref{def flatness}. Now, again using the flatness assumption in Definition \ref{def flatness}, we compute 
$$||Q||^2_{L^2}= \sum_{\xi_2}|d_{\xi_2}|^2\lesssim  u(N)^2N^{-1}\lesssim N^{-1/2}.$$
Therefore, in light of Lemma \ref{ERKA}, we have shown the following claim: 
\begin{claim}
	\label{claim: bound on Vf norm} There exists a density one subsequence $S'\subset S$ such that for all sufficiently large $\lambda\in S'$ the following holds: suppose that $f_{\lambda}$  in \eqref{function} is flat then we have 
$$ 	||P||^2_{L^2(\mathbb{T})}> 1/2,$$
where $P$ is as in \eqref{eq: L^2-norm}.
\end{claim}
We will also need the following simple consequence of Lemma \ref{semi-correlations}: 
\begin{claim}
	\label{claim:  E2}
	Let $p\geq 2$ be a positive even integer, define the set $V_{2,\lambda}=V_2:=\{2\xi_2: \xi=(\xi_1,\xi_2), \hspace{2mm}|\xi|^2=\lambda\} \cup \{0\}$ and let $D(V_2)$ be as in \eqref{definition}. Then, there exists a density one subsequence $S'=S'(p)\subset S$ such that,  for all sufficiently large $\lambda\in S'$,  $D(V_2)$ is a $\Lambda(p)$-system with constant $C_0=C_0(p)$ independent of $V_2$.  
\end{claim}
\begin{proof}
We select the subsequence $S'$ so that the conclusion of  Lemma \ref{semi-correlations} holds for all $\ell \leq 2p$.  Moreover, in light of Lemma \ref{ERKA}, we may assume that $V_2$ has more than $p$ elements. Therefore, upon observing that $V_2$ is a symmetric set, Claim \ref{no non-trivial sol} implies that, in order to prove Claim \ref{claim:  E2}, it is enough to show the following: given $p$ elements  $n_{1},...,n_{p}\in D(V_2)$ satisfying 
\begin{align}
	n_{1}+...+n_{p}=0, \label{extra bs}
\end{align}
we have only trivial solutions to \eqref{extra bs}, that is $n_{1}= -n_{2}$..., up to permutations. 

First we observe that, since $0\in V_2$, every $n_i\in D(V_2)$ is either a single projection\footnote{in the sense of being the projection of $\xi$ onto the second coordinate.} $n_i=2\xi^i_2$ or the difference of two (different) projections $n_i=2(\xi^i_2-\eta^i_2)$ for $\xi^i_2\neq \eta^i_2$. In the former case, we say, for convenience, that $n_i$ is of Type I.  Suppose that none of the $n_{i}$'s is of type I, then \eqref{extra bs} reads 
$$ 2( \xi^1_2-\eta^1_2 +... + \xi^p_2-\eta^p_2 )=0.$$
Thus,  all solutions to \eqref{extra bs} are trivial by the choice of $S'$. Now, suppose that at least one of the $\{n_{i}\}$ is of Type I. If the number of Type I integers is odd then the conclusion of Lemma \ref{semi-correlations} asserts that there are no solutions. Thus, there must be an even number of Type I integers among the $\{n_{i}\}$. Hence, all solutions to \eqref{extra bs} are trivial again by the choice of $S'$.

\end{proof}
\subsection{Concluding the proof of Proposition \ref{prop: level sets estimates}}
We are now ready to prove Proposition \ref{prop: level sets estimates}: 
\begin{proof}[Proof of  Proposition \ref{prop: level sets estimates}] 
We choose the subsequence $S'$ to be the intersection of three sub-sequences $S_1,S_2$ and $S_3$ as follows. We pick $S_1=S_1(q)$ so that the conclusion of Lemma \ref{semi-correlations} holds for all $\ell\leq 10p$ for some $p=p(q)$ to be chosen later. We choose $S_2$ so that the conclusion of  Claim \ref{claim: bound on Vf norm} holds and  $S_3$  so that the conclusion of Claim \ref{claim:  E2} holds. Having prescribed $S'$, we begin the proof of Proposition \ref{prop: level sets estimates}.

	Given $t\geq 0$ and an integer $q>0$, let $A$ be as in \eqref{def A_{t,q}}. Observe that 
	\begin{align} 
	\label{eq: important}	\vol(A)= \vol (A \cap E) + \vol (A \cap F),
		\end{align} 
	where 
	$$E=E_{t,q}:= \{x=(x_1,x_2)\in \mathbb{T}^2: \log|P(x_2)|> -t^{1/q} \}$$
	and $F$ is the complement of $E$ and $P(\cdot)$ is as in \eqref{eq: L^2-norm}.  First, we are going to bound the first term on the r.h.s. of \eqref{eq: important}. Writing $\rho(\cdot)$ for the (normalized) Lebesgue measure on $\mathbb{T}$, by Fubini we have  
	\begin{align}
	 \vol (A \cap E) \leq \int_{\tilde{E}} \rho\left(x_1: \log |f(x_1,x_2)|\leq -t^{1/q} \right) dx_2,  \nonumber
	\end{align}
where $\tilde{E}$ is the projection of $E$ onto the second coordinate. Using the trivial bound $\rho(\tilde{E})\leq 1$ and in light of the notation introduced in \eqref{def; f^1}, we deduce
 	\begin{align}
 	\vol (A \cap E) &\leq \sup_{x_2\in \tilde{E}} \rho\left(x_1: \log \left|f(x_1,x_2)\right|\leq- t^{1/q}\right)= \sup_{x_2\in \tilde{E}} \rho\left(x_1: \log \left| (H_{x_2}f)(x_1)\right|\leq- t^{1/q} \right) \nonumber \\
 	 &= \sup_{x_2\in \tilde{E}} \rho\left(x_1: \log \left|(H_{x_2}f)(x_1)\right|- \log( |P(x_2)|^{1/2})\leq- t^{1/q}- \log( |P(x_2)|^{1/2})  \right) \nonumber \\
 	 &\leq \sup_{x_2\in \tilde{E}} \rho\left(x_1: \log \left|\frac{(H_{x_2}f)(x_1)}{P^{1/2}}\right|\leq- \frac{t^{1/q}}{2}  \right) . \label{extra steps}
 \end{align}
To estimate the r.h.s. of \eqref{extra steps}, we wish to apply Corollary \ref{cor: Nazarov} to the function $$g(x_1)=g_{x_2}(x_1):= \frac{H_{x_2}f(x_1)}{P^{1/2}(x_2)}.$$ Observe that, for $x_2\in \tilde{E}$, $g$ is well-defined and, by \eqref{eq: L^2-norm}, we also have 
$$||g||_{L^2}=1.$$ 
 Therefore, in order to apply Corollary \ref{cor: Nazarov}, it is enough to verify the assumptions of Theorem \ref{Thm: Logintegrability} for the set $V_1=V_{1,\lambda}= \{\xi_1: \xi=(\xi_1,\xi_2), \hspace{3mm} |\xi|^2=\lambda \}$, which is a symmetric set.  Let $R(V_1)$ and $D(V_1)$ be as in \eqref{definition}, then, by Lemma \ref{semi-correlations} and  Claim \ref{no non-trivial sol},  $R(V_1)=3$   and $D(V_1)$ is a $\Lambda(p)$-system with constant $C_0=C_0(p)$ independent of $V_1$. Thus, we are in the position of applying Corollary \ref{cor: Nazarov} with (say) $p= 8q$ and $\varepsilon= 1/q$, to find some constant $C=C(q)$, independent of $x_2$, such that 

	\begin{align}
	\label{eq: important 1}
\sup_{x_2 \in \tilde{E}}	\rho\left(x_1: \log \left|\frac{(H_{x_2}f)(x_1)}{P^{1/2}}\right|\leq -\frac{t^{1/q}}{2} \right)\leq C t^{-1-1/(3q)}.
	\end{align} 

We are now going to bound the second term on the r.h.s. of \eqref{eq: important}. Observe that
	\begin{align}
	\vol (A \cap F) \leq \vol(F)\leq  \rho\left(x_2: \log|P(x_2)|\leq -t^{1/q} \right). \label{extra steps 2}
\end{align}
Similarly to the above argument, we are going to use Corollary \ref{cor: Nazarov} to bound the r.h.s. of \eqref{extra steps 2}. Thanks to claims \ref{claim: bound on Vf norm} and \ref{claim:  E2}, we may apply Corollary \ref{cor: Nazarov} to the function $P$ with   $p= 8q$ and $\varepsilon= 1/q$ to see that there  exist constants $\lambda_0=\lambda_0(q)$ and $\tilde{C}=\tilde{C}(q)$ such that  for all $\lambda>\lambda_0$ in $S'$, we have 
	\begin{align}
	\label{eq: important 2}
 \rho\left(x_2: \log|P(x_2)|\leq -t^{1/q} \right)\leq \tilde{C} t^{-1-1/(3q)}.
\end{align} 
Hence, Proposition \ref{prop: level sets estimates}, with $\alpha= 1/(3q)$, follows by combining \eqref{eq: important}, \eqref{extra steps}, \eqref{eq: important 1}, \eqref{extra steps 2} and \eqref{eq: important 2}. 
\end{proof}
\section{Moments of $\mathcal{L}(F_x)$}
\label{sec:moments}
The aim of this section is to apply Proposition \ref{prop: level sets estimates} to show that arbitrarily high moments of  $\mathcal{L}(F_x)$, with $F_x$ as in \eqref{F_x}, are integrable. In  particular, this will show that the \textquotedblleft bad\textquotedblright set of $x\in \mathbb{T}^2$, coming from Proposition \ref{prop: convergence in distribution}, does not significantly contribute to the moments of $\mathcal{L}(F_x)$, leading to the fundamental Proposition \ref{prop: convergence of moments} below. The main result of this section is the following:
\begin{prop}
	\label{prop: integrability nodal length}
	Let $q\geq 1$ be an integer.  There exists a density one subsequence $S'=S'(q)\subset S$ and some constant $C=C(q)>0$ such that for all $\lambda\in S'$ the following holds: suppose that $f_{\lambda}$  in \eqref{function} is flat, then we have 
	$$ \int_{\mathbb{T}^2} \mathcal{L}(F_x)^{q} dx \leq C,  $$
	where $F_x$ is as in \eqref{F_x}. 
\end{prop}
The following corollary is a direct consequence of Proposition \ref{prop: integrability nodal length}: 
\begin{cor}
	\label{cor: cor log integrability}
	Let $q\geq 1$ be an integer. There exists a density one subsequence $S'=S'(q)\subset S$ such that for all $\lambda\in S'$ and all fixed balls $B\subset \mathbb{T}^2$ the following holds: there exists some constant $C=C(q,B)>0$ such that if  $f_{\lambda}$ in \eqref{function} is flat then
		 	$$ \frac{1}{\vol B}\int_{B} \mathcal{L}(F_x)^{q} dx \leq C,  $$
	where $F_x$ is as in \eqref{F_x}. 			
\end{cor}

Thanks to Corollary \ref{cor: cor log integrability} we can \textquotedblleft upgrade\textquotedblright the convergence in distribution in Proposition \ref{prop: convergence in distribution} to convergence of expectations.  Formally, we will need the following well-known fact about uniform integrability \cite[Theorem 3.5]{BI}:
\begin{lem}\label{th:DCTh}
	Let $X_n$ be a sequence of random variables such that $X_n\overset{d}{\rightarrow} X$, that is, convergence in distribution. Suppose that there exists some $\alpha>0$ such that $\mathbb{E}[|X_n|^{1+\alpha}]\le C<\infty$ for some $C>0$, uniformly for all $n\geq 1$. Then,
	\begin{align}
		&\mathbb{E}[X_n]\to \mathbb{E}[ X] &n\rightarrow \infty. \nonumber
	\end{align}
\end{lem}

We are finally ready to state (and prove) the main consequence of Proposition \ref{prop: integrability nodal length}:
\begin{prop}
	\label{prop: convergence of moments}
 There exists a density one subsequence $S'\subset S$ such that the following holds: let $\{f_{\lambda}\}_{\lambda\in S'}$ be a sequence of flat eigenfunctions  with limiting measure  $\mu$ in the sense of \eqref{spectral measure} and $B\subset\mathbb{T}^2$ be a fixed ball or $B=\mathbb{T}^2$, then
	\begin{align}
		\nonumber &\frac{1}{\vol B}\int_{B}\mathcal{L}(F_x)dx\rightarrow \mathbb{E}[\mathcal{L}(F_{\mu})]&\lambda\rightarrow \infty,
	\end{align}	
where $F_x$ is as in \eqref{F_x}.  
\end{prop}
\begin{proof}
Let $S'$ be the intersection of the sub-sequence given by Proposition \ref{prop: convergence in distribution} and the sub-sequence in Corollary \ref{cor: cor log integrability} applied with (say) $q=2$. Then, under the assumptions of Proposition \ref{prop: convergence of moments}, we have
$$\mathcal{L}(F_x^B)\overset{d}{\longrightarrow} \mathcal{L}(F_{\mu}).$$
Now, Proposition \ref{prop: convergence of moments} follows from Corollary \ref{cor: cor log integrability} via Lemma \ref{th:DCTh}.
\end{proof}
The rest of this section is dedicated to the proof of Proposition \ref{prop: integrability nodal length}. 
\subsection{Proof of Proposition \ref{prop: integrability nodal length}}
\label{doubling index sec}
In this section, given a box $B\subset \mathbb{R}^n$ and $r>0$, we write $rB$ for the concentric box of $r$-times the side. It is a well-known fact, see \cite[Proposition 6.7]{DF},  that the nodal length of Laplace eigenfunction, on real analytic manifolds, can be bounded in terms of the doubling index. Given a (say) $C^3$ function $g: 3B\rightarrow \mathbb{R}$, the doubling index of $g$ on the box $B$ is defined by 
\begin{align}
	\label{def: doubling index}
	N_g(B):=  \log \frac{\sup_{2B}|g|}{\sup_B |g|} + 1.
\end{align}
 We added $1$ in \eqref{def: doubling index} to ensure that the doubling index is strictly greater than zero. 
 
  In order to prove Proposition \ref{prop: integrability nodal length}, we will use following lemma, see \cite[Lemma 2.6.1]{LMlecturenotes} and \cite[Proposition 6.7]{DF}, to control the (local) nodal length:
 \begin{lem}
	\label{doubling index}
	Let $\tilde{B} \subset \mathbb{R}^{3}$ be the unit box, suppose that $h: 3\tilde{B}\rightarrow \mathbb{R}$ is an harmonic function $(\Delta h=0)$, then 
	\begin{align}
		\mathcal{V}\left(h, \tilde{B}\right)\lesssim  N_h(2\tilde{B}),   \nonumber
	\end{align}
where $\mathcal{V}\left(h, \tilde{B}\right)= \mathcal{H}^{2}(\{x\in \tilde{B}; h(x)=0 \})$. 
\end{lem}
Lemma \ref{doubling index} has the following direct consequence: 
\begin{lem}
	\label{lem: control nodal length} 
	Let $F_x$ be as in \eqref{F_x} and $q\geq 1$ be a fixed integer. We have the following bound: 
	$$\mathcal{L}(F_x)^{q}\lesssim_q N_{F_x}(B)^q,$$
	where $B=B(2)\subset \mathbb{R}^2$ is the box of side $2$ centered at zero. 
\end{lem}
\begin{proof}
First, observe that, by trivially extending the domain of $F_x$, we may assume that $F_x$  is well-defined on the box of side (say) $20$ centered at $0$. 	Let $B= [-1,1]^2$ and $\tilde{B}=[-1,1]^2\times[-1,1]$. Then the \textquotedblleft harmonic lift\textquotedblright of $F_x$
$$h(y,s) := F_x(y)e^{2\pi s}: 3\tilde{B}\rightarrow\mathbb{R},$$  
 is harmonic, that is $\Delta h=0$. Therefore, Lemma \ref{doubling index} gives 
 \begin{align}
\nonumber
 	\mathcal{V}\left(h, \frac{1}{2}\tilde{B}\right)= \mathcal{H}^{2}(\{x\in 2^{-1}\tilde{B}; h(x)=0 \}) \lesssim  \log \frac{\sup_{2\tilde{B}}|h|}{\sup_{\tilde{B}} |h|} +1.   
 \end{align} 
Observe that if $F_x$ vanishes at some point $y$ then $h$ vanishes on the line $\{y\}\times[-1,1]$, thus
 $$\mathcal{L}(F_x)\lesssim 	\mathcal{V}\left(h, \frac{1}{2}\tilde{B}\right).$$
Moreover, since $\sup_{2\tilde{B}}|h|= e^{4\pi } \sup_{2B} |F_x|$ and $\sup_{\tilde{B}}|h|= e^{2\pi } \sup_{B} |F_x|$ , we also have 
 $$	\mathcal{V}\left(h, \frac{1}{2}\tilde{B}\right)\lesssim 1 + \log \frac{\sup_{2B}|F_x|}{\sup_{B} |F_x|}. $$
 Hence, Lemma \ref{lem: control nodal length} follows from the inequality $(X+Y)^{q}\lesssim_{q} |X|^{q}+ |Y|^{q}$.
\end{proof}
Lemma \ref{lem: control nodal length}, up to scaling factors, reduces  Proposition \ref{prop: integrability nodal length} to \eqref{anti-concentration} which, we are now going to show, follows from Proposition  \ref{prop: level sets estimates}. In the proof of \eqref{anti-concentration}, we will need a standard consequence of the elliptic estimates for harmonic functions \cite[Page 332]{Ebook}.  The elliptic estimates state that any $L^p$-norm, for $2\leq p\leq \infty$, of a harmonic function in a ball/box $B$ is bounded by its $L^2$-norm on (say) $\tfrac{3}{2}B$. More precisely, we have the following fact:
\begin{lem}
	\label{elliptic regularlity}
	Let $F_x$ be as in \eqref{F_x}, $h(y,s)= F_x(y)\cdot e^{2\pi s}$ be as in (the proof of) Lemma \ref{lem: control nodal length}, $B=[-1,1]^2$ and $\tilde{B}=[-1,1]^2\times [-1,1]$. Then, we have 
	$$ \left(\sup_{2B} |F_x|\right)^2\leq  \left(\sup_{2\tilde{B}}|h|\right)^2 \lesssim ||h||^2_{L^2(3\tilde{B})}. $$
\end{lem}  
We are finally ready to prove Proposition \ref{prop: integrability nodal length}. 
\begin{proof}[Proof of Proposition \ref{prop: integrability nodal length}]

Given $q\geq 1$, let $S'\subset S$ be given by Proposition \ref{prop: level sets estimates}. Moreover,  by trivially extending the domain of $F_x$, we may assume that $F_x$  is well-defined on the box of side (say) $20$ centered at $0$.  Thanks to Lemma \ref{lem: control nodal length}, it is enough to show that there exists some constant $C=C(q)$ such that 
\begin{align}
\label{eq: 1.1}
\int_{\mathbb{T}^2}\left(\log \frac{\sup_{2B}|F_x|}{\sup_{B} |F_x|}\right)^q dx \leq C,
\end{align}
where $B=B(2)\subset \mathbb{R}^2$ is the box of side $2$ centered at zero. Using the inequality  $(X+Y)^{q}\lesssim_{q} X^{q}+ Y^{q}$ it is enough to show that 
$$\int_{\mathbb{T}^2}\left|\log \sup_{2B}|F_x|\right|^q dx \leq C , \hspace{10mm}\int_{\mathbb{T}^2}\left|\log \sup_{B}|F_x|\right|^q dx \leq C. $$
 We are just going to show the first claimed bound, the proof of the second one being identical. First, as in the proof of Proposition \ref{log-integrability}, recalling the notation in section \ref{sec: notation}, we may write 
 \begin{align}
 	\label{to prove} \int_{\mathbb{T}^2}\left|\log \sup_{2B}|F_x|\right|^q dx = \int_{10}^{\infty} \vol\left(  \log \sup_{2B}|F_x|> t^{1/q}\right) + \vol\left(  \log \sup_{2B}|F_x|\leq- t^{1/q}\right) dt +O(1).
 	\end{align}
  Since $\sup_{2B}|F_x|\geq |f(x)|$, the second term on the r.h.s. of \eqref{to prove} is bounded by some constant depending on $q$ only by Proposition \ref{prop: level sets estimates}. Therefore, in order to prove Proposition \ref{prop: integrability nodal length}, it is enough to show that the first term on the r.h.s. of \eqref{to prove} is bounded   by some constant depending on $q$ only. 
  
  Writing $h$ to be the harmonic lift of $F_x$ introduced in the proof of Lemma \ref{lem: control nodal length}, $B=[-1,1]^2$ and $\tilde{B}=[-1,1]^2\times [-1,1]$, Lemma \ref{elliptic regularlity} gives
$$  \left(\sup_{2B} |F_x|\right)^2\leq  \left(\sup_{2\tilde{B}}|h|\right)^2 \lesssim ||h||^2_{L^2(3\tilde{B})}\lesssim \int_{3B}|F_x(y)|^2dy.$$
Since 
 $$\int_{\mathbb{T}^2} ||F_x||^2_{L^2(3B)}= \sum_{\xi,\eta}a_{\xi}\overline{a_{\eta}} \int_{\mathbb{T}^2} e((\xi-\eta)x)dx \int_{3B} e((\xi-\eta)\lambda^{-1/2}y)dy= O(1),$$
 Chebyshev's inequality gives  $$\vol\left( \sup_{2B} |F_x|> \exp(t^{1/p})\right) \lesssim \exp(-2t^{1/p}).$$
 Therefore the first term on the r.h.s. of \eqref{to prove} is bounded. This concludes the proof of Proposition \ref{prop: integrability nodal length}. 
\end{proof}

\section{Concluding the proofs of the main results}
\subsection{Nodal length of Gaussian random fields}
In this section we collect a few preliminary results towards the proofs of Theorem \ref{thm 1.1} and Theorem \ref{thm 1}. 

\begin{lem}
	\label{expectation}
Let $\mu$ be a symmetric probability measure supported on $\mathbb{S}^1$, and not supported on a line. Then we have
\begin{align}
	\mathbb{E}[\mathcal{L}(F_{\mu})]=c_1 \cdot (2\pi), \nonumber
\end{align}
where
\begin{align}
	c_1= \frac{1-|\widehat{\mu}(2)|^2}{2^{5/2}\pi}\int_{0}^{2\pi}\frac{1}{(1-\alpha\cos(2\theta)-\beta \sin(2\theta))^{3/2}}d\theta, \nonumber
\end{align}
and $$\widehat{\mu}(2)=\int_{\mathbb{S}^1} z^2d\mu(z)=: \alpha+i\beta.$$
\end{lem}
The proof of Lemma \ref{expectation} follows by a standard use of the Kac-Rice formula \cite[Theorem 6.2]{AW}.
\begin{proof}
	We write $F=F_{\mu}$. Since $\mu$ is not supported on a line, $(F,\nabla F)$  is non-degenerate, thus we  apply the Kac-Rice formula \cite[Theorem 6.2]{AW} to see that 
	\begin{align}
		\mathbb{E}[\mathcal{L}(F_{\mu}, B)]= \int_{ B}\mathbb{E}\left[|\nabla F(y)|| F(y)=0\right]\phi_{F(y)}(0)dy, \label{kac-rice} 
	\end{align}
	where $\phi_{F(y)}(\cdot)$ is the density of the random variable $F(y)$. Since $F(y)$ is Gaussian with mean zero and variance $1$, $\phi_{F(y)}(0)= 1/\sqrt{2\pi}$. As $F$ and $\nabla F$ are independent (this can be seen directly differentiating $\mathbb{E}[F(y)^2]=1$) and using stationarity, we also have 
	$$\mathbb{E}\left[|\nabla F(y)|| F(y)=0\right]=\mathbb{E}[|\nabla F(y)|]= \mathbb{E}[|\nabla F(0)|] .$$ 
	 Thus,  (\ref{kac-rice}) simplifies to 
	\begin{align}
		\mathbb{E}[\mathcal{L}(F_{\mu})]=\frac{1}{\sqrt{2\pi}} \cdot \mathbb{E}[|\nabla F(0)|]= \frac{ 2\pi}{\sqrt{2\pi}} \mathbb{E}[(2\pi)^{-1}|\nabla F(0)|]  . \label{16}
	\end{align}
	
	Now, we compute the covariance matrix of $\nabla F$. First we write 
	$$\widehat{\mu}(2)= \alpha +i\beta:=\int_0^1 \cos(2\theta)d\mu(e(\theta))+i\int_0^1 \sin(2\theta)d\mu(e(\theta))$$ and
	$$\mathbb{E}[F(x)F(y)]=\int_{\mathbb{S}^2}e(\langle x-y, s\rangle) d\mu(s).$$
	By using the relations $\cos(2\theta)=2\cos^2(\theta)-1=1-2\sin^2(\theta)$ and $\sin(2\theta)=2\sin(\theta)\cos(\theta)$ and writing $s=(s_1,s_2)$ and $x=(x_1,x_2)$, we have
	\begin{align}
		&(2\pi)^{-2}\mathbb{E}[\partial_{x_1}^2F(x)F(y)]\arrowvert_{x=y}=\int_{\mathbb{R}^2}s_1^2d\mu(s)=\int_{0}^1\cos^2(\theta)d\mu(e(\theta))=\frac{1}{2}+ \frac{\alpha}{2} \nonumber\\
		&(2\pi)^{-2}\mathbb{E}[\partial_{x_2}^2F(x)F(y)]\arrowvert_{x=y}=\int_{\mathbb{R}^2}s_2^2d\mu(s)=\int_{0}^1\sin^2(\theta)d\mu(e(\theta))=\frac{1}{2}- \frac{\alpha}{2} \nonumber \\
		&(2\pi)^{-2}\mathbb{E}[\partial_{x_1}\partial_{y_2}F(x)F(y)]\arrowvert_{x=y}=\int_{\mathbb{R}^2}s_1s_2d\mu(s)= \int_{0}^1\cos(\theta)\sin(\theta)d\mu(e(\theta))= \frac{\beta}{2}. \label{104}
	\end{align}
	Therefore, the covariance matrix of $(2\pi)^{-1}\nabla F$ is
	\begin{align}
		&L=\begin{bmatrix}
			\frac{1}{2}+ \frac{\alpha}{2} & \frac{\beta}{2} \\
			\frac{\beta}{2} & \frac{1}{2}- \frac{\alpha}{2} 
		\end{bmatrix}  & \det(L)= \frac{1}{4}\left( 1- \alpha^2-\beta^2\right)\label{covariance1}.
	\end{align}

	Since $(2\pi)^{-1}\nabla F(0)$ is a bi-variate Gaussian with mean $0$ and covariance $L$, given in (\ref{covariance1}), we have 
	
	\begin{align}
		&\mathbb{E}[|(2\pi)^{-1}\nabla F(0)|]= \nonumber \\
		&=\frac{1}{\pi{(1-\alpha^2-\beta^2)^{1/2}}}\int_{\mathbb{R}^2}\sqrt{x_2+y^2}\exp\left(-\frac{x_2(1-\alpha)+ y^2(1+\alpha)- 2\beta xy}{(1-\alpha^2-\beta^2)}\right)dxdy. \label{extra}
	\end{align} 
 Finally, by passing to polar coordinates in  \eqref{extra} we have:
	\begin{align}
		&\mathbb{E}[|(2\pi)^{-1}\nabla F|]= \nonumber \\
		&=\frac{1}{\pi(1-\alpha^2-\beta^2)^{1/2}}\int_{0}^{2\pi}d\theta\int_{0}^{\infty}r^2\exp\left(-\frac{r^2}{(1-\alpha^2-\beta^2)}\left(1-\alpha\cos(2\theta) -\beta \sin(2\theta)\right)\right)dr .\nonumber
	\end{align}
	Substituting $r=(\eta y)^{1/2}$, where $\eta=\eta(\theta)=(1-\alpha\cos (2\theta)-\beta\sin(2\theta))^{-1}(1-\alpha^2-\beta^2)$, we deduce 
	\begin{align}
		\mathbb{E}[|(2\pi)^{-1}\nabla F|]&= \frac{1}{2\pi(1-\alpha^2-\beta^2)^{1/2}}\int_{0}^{2\pi}\eta^{3/2}d\theta\int_{0}^{\infty}y^{1-1/2}e^{-y}dy  \nonumber \\
		&= \frac{1}{2\pi}\Gamma\left(1+\tfrac{1}{2}\right)(1-\alpha^2-\beta^2)\int_{0}^{2\pi}\frac{1}{(1-\alpha\cos(2\theta)-\beta\sin(2\theta))^{3/2}}d\theta. \label{15}
	\end{align} 
	As $\Gamma(3/2)= \sqrt{\pi}/2$, Lemma \ref{expectation} follows from  \eqref{16} and \eqref{15}. 
	
\end{proof}
We will also need the following lemma: 
\begin{lem}
	\label{locality nodal length}
	 There exists a density one subsequence $S'\subset S$ such that the following holds: let $B\subset \mathbb{T}^2$ be a ball of radius $r>0$, suppose that $f_{\lambda}$ in \eqref{function}, with $\lambda\in S'$, is flat then we have
$$ \mathcal{L}(f_{\lambda},B)= \lambda^{1/2}\int_{B} \mathcal{L}(F_x)dx + O\left(r^{1/2}\lambda^{1/4}\right),$$
where $F_x$ is as in \eqref{F_x}. 
\end{lem}
\begin{proof}
Let the postulated subsequence be given by Proposition \ref{prop: integrability nodal length} with $q=2$. Let us write $B=B(z,r)=B(r)$ for some $z\in \mathbb{T}^2$ and $r>0$ and  observe that 
 \begin{align}
 \lambda^{1/2}\int_{B(r- 1/\lambda^{1/2})} \mathcal{L}(F_x)dx\leq \mathcal{L}(f_{\lambda},B)\leq   \lambda^{1/2}\int_{B(r+ 1/\lambda^{1/2})} \mathcal{L}(F_x)dx \label{claim locality},
 \end{align}
 Indeed, writing $r'=1/\lambda^{1/2}$, by definition of $\mathcal{L}(\cdot)$ and Fubini\footnote{Note that, since $f^{-1}(0)$ is the zero set of a real-analytic function, the Hausdorff measure coincide with the Lebesgue measure so we may apply Fubini.}, we have 
\begin{align} 
	\int_{B(r)}{\mathcal{L}(f, B(x, r-r'))} dx&=	\int_{B(r-r')}\int_{B(r)}\mathds{1}_{B(x,r')}(y)\mathds{1}_{f^{-1}(0)}(y) d\mathcal{H}(y)dx. \nonumber \\
	&=\int_{B(r)} \mathds{1}_{f^{-1}(0)}(y)\vol\left(B(y,r')\cap B(R)\right) d\mathcal{H}(y), \nonumber
\end{align}
where $\mathds{1}$ is the indicator function and $\mathcal{H}$ the Hausdorff measure. Thus \eqref{claim locality} follows from  the scaling property $\mathcal{L}(f, B(x,r'))=  \lambda^{-1/2} \mathcal{L}(F_x)$, upon noticing
$$
\mathds{1}_{B(r-r')}\le\dfrac{\vol\left(B(\cdot,r')\cap B(1)\right)}{\vol{B(r')}}\le \mathds{1}_{B(r+r')}.
$$
Finally, by Proposition \ref{prop: integrability nodal length} with $q=2$ and the Cauchy-Schwarz inequality, for a density one subsequence of $\lambda \in S$, we have 
\begin{align}
\left(\int_{B(r)}- \int_{B(r\pm 1/\lambda^{1/2})}\right) \mathcal{L}(F_x)dx \lesssim  \frac{r^{1/2}}{\lambda^{1/4}}. \label{5.2.1}
\end{align} 
Hence Lemma \ref{locality nodal length} follows from \eqref{claim locality} and \eqref{5.2.1}. 
\end{proof}
\subsection{Proof of theorems \ref{thm 1.1} and \ref{thm 1}}
In order to prove Theorem \ref{thm 1.1}, we will need the following fact, see \cite{EH99} and \cite{KK77}. 
\begin{lem}
\label{extra lemma}	There exists a density one subsequence of $S'\subset S$ such that, for any sequence of $f_{\lambda}$ in \eqref{function} with $\lambda\in S'$ and satisfying the assumption \eqref{Bourgain}, we have 
	$$\mu_f\rightarrow \rho,$$
	where $\mu_f$ is as in \eqref{spectral measure1} and $\rho$ is the uniform measure on the unit circle $\mathbb{S}^1$.  
\end{lem}
We are finally ready to carry out the proof of theorems \ref{thm 1.1} and \ref{thm 1}:
\begin{proof}[Proof of Theorem \ref{thm 1.1}]
 Let $S'\subset S$ be such that the conclusions of   Lemma \ref{locality nodal length}, Lemma \ref{extra lemma} and Proposition \ref{prop: convergence of moments} hold.  Let $B$ be a fixed ball of radius $r>0$ or $B=\mathbb{T}^2$, then by Lemma \ref{locality nodal length} and Proposition \ref{prop: convergence of moments},  we have 
\begin{align} 
\label{equation final}
 \mathcal{L}(f,B)&= \lambda^{1/2} \vol (B)\cdot \frac{1}{\vol(B)} \int_B \mathcal{L}(F_x)dx   +O(\lambda^{1/4}) \nonumber \\
 &= \lambda^{1/2}\vol(B)\mathbb{E}[\mathcal{L}(F_{\rho})](1+o_{\lambda\rightarrow \infty}(1)),
 \end{align}
where, thanks to Lemma \ref{extra lemma}, $\rho$ is the uniform measure on the unit circle $\mathbb{S}^1$. Using the explicit formula in Proposition \ref{expectation} with $\mu=\rho$ (that is $\alpha=\beta=0$) we obtain
 $$ \mathbb{E}[\mathcal{L}(F_{\rho})]= \frac{2\pi}{2\sqrt{2}}.$$
 Hence,  Theorem \ref{thm 1.1} follows from \eqref{equation final}.

\end{proof}

\begin{proof}[Proof of Theorem \ref{thm 1}]
	The proof  is very similar to the proof of Theorem \ref{thm 1.1}. Indeed, for a density one  subsequence of $\lambda \in S$, the asymptotic law \eqref{equation final} holds. Now, the r.h.s. of \eqref{equation final} can be computed via Proposition \ref{expectation} concluding the proof of Theorem \ref{thm 1}. 
	\end{proof}

\section*{Acknowledgments}
The author would like to thank Peter Sarnak for asking the question about finding the asymptotic nodal length which gave rise to this work. We also thank Igor Wigman for the  many discussions and Alon Nishry for pointing out the work of Nazarov and reading the first draft of the article. We also thank the anonymous referees for their comments that improved the readability of the paper. This work was supported by the Engineering and Physical Sciences Research Council [EP/L015234/1], The EPSRC Centre for Doctoral Training in Geometry and Number Theory (The London School of Geometry and Number Theory), University College London, the ISF Grant
1903/18 and the BSF Start up Grant no. 20183
\appendix

\section{Log-integrability}
\label{sec: proof of Nazarov's Theorem}
For the sake of completeness, in this section we provide the proof of Theorem \ref{Thm: Logintegrability}. The proof is based on \cite{N93,Nun}, we claim no originality. 
\subsection{Proof of Theorem \ref{Thm: Logintegrability}}
The main ingredient in the proof of Theorem \ref{Thm: Logintegrability} is the following Lemma, which we will prove in section \ref{proof of Lemma spreading} below, see \cite[Corollary 3.5]{N93} and \cite{Nun}.  
\begin{lem}[Spreading Lemma]
	\label{Spreading lemma}
	Let $V= \{n_i\}_i\subset \mathbb{Z}$ be a set such that $R(V)<\infty$ with $R(V)$ as in \eqref{definition}. Moreover, let $U\subset \mathbb{T}$ be a positive measure set with $\rho(U)\leq 4R(V)/(4R(V)+1)$, where $\rho(\cdot)$ is the Lebesgue measure on $\mathbb{T}$. Suppose that there exists some integer $m\geq 1$ such that 
	$$ \frac{4}{\rho(U')^2}\sum_{n_i\neq n_j} \left| \widehat{ \mathds{1}}_{U'}(n_i-n_j)\right|^2\leq m+1,$$
	for all subsets $U'\subset U$ of measure $\rho(U')\geq \rho(U)/2$, where $\mathds{1}_{U'}$ is the indicator function of the set $U'$. Then, there exists a set $U_1\supset U$ such that 
	\begin{enumerate}
		\item The measure of $U_1$ satisfies
		$$\rho(U_1\backslash U)\geq \frac{\rho(U)}{4m}.$$
		\item Uniformly for all $g\in L^2(\mathbb{T})$ with $\Spec(g)\subset V$, we have 
		$$ \int_{U_1} |g(x)|^2dx\leq \left( C \frac{m^5}{\rho(U)^2}\right)^{3m}\int_{U} |g(x)|^2dx,$$
		for some absolute constant $C>0$.
	\end{enumerate}
\end{lem}
In order to apply Lemma \ref{Spreading lemma}, we will need the following two claims: 
\begin{claim}
	\label{choice of m}
	Under the assumptions of Theorem \ref{Thm: Logintegrability} and maintaining the same notation, the integer $m>0$ in Lemma \ref{Spreading lemma} can be taken to be 
	\begin{align}
		m= \left[ \frac{C_0^2 R(V)}{\rho(U')^{\frac{2}{p}}}\right]=: \left[ \frac{G}{\rho(U')^{\frac{2}{p}}}\right] \nonumber
	\end{align}
	where $[\cdot]$ is the integer part.
\end{claim}
\begin{proof}
	By definition of the $L^2(\mathbb{T})$ norm, we have 
	\begin{align}
		\left(\sum_{n_i\neq n_j} \left| \widehat{ \mathds{1}}_{U'}(n_i-n_j)\right|^2\right)^{1/2}\leq R(V)^{1/2} \left(\sum_{r\in D(V)} \left| \widehat{ \mathds{1}}_{U'}(r)\right|^2\right)^{1/2} \nonumber \\
		= R(V)^{1/2} \sup \left\{\left| \int_{U'} \overline{h}(x)dx\right| : ||h||_{L^2(\mathbb{T})}\leq 1, \hspace{3mm} \Spec(h)\subset D(V)\right\}, \label{B.2}
	\end{align}
with $D(V)$ as in \eqref{definition}.	Now, since $D(V)$ is a $\Lambda(p)$-system, we can bound the right hand side of \eqref{B.2} using H\"{o}lder's inequality as follows:
	\begin{align}
		\nonumber
		R(V)^{-1/2}\text{r.h.s}\eqref{B.2}\leq \rho(U')^{1-\frac{1}{p}} \sup \left\{||h||_{L^{p}(\mathbb{T})} : ||h||_{L^2(\mathbb{T})}\leq 1, \hspace{3mm} \Spec(h)\subset D(V)\right\}\nonumber \\
		 \leq C_0\rho(U')^{1-\frac{1}{p}} \nonumber,
	\end{align}
	for some constant $C_0=C_0(V,p)>0$ as in \eqref{def tildeC}.  Therefore, in light of \eqref{B.2}, we obtain 
	$$\frac{4}{\rho(U')^2}\sum_{n_i\neq n_j} \left| \widehat{ \mathds{1}}_{U'}(n_i-n_j)\right|^2\leq C_0^2 R(V) \rho(U')^{-\frac{2}{p}}, $$
	as required.
\end{proof}

\begin{claim}
	\label{trivial range}
Under the assumptions of Theorem \ref{Thm: Logintegrability} and maintaining the same notation,	let $g\in L^2(\mathbb{T})$ with $\Spec(g)\subset V=\{n_i\}_i$, and $U\subset \mathbb{T}$ be a measurable subset. If $\rho(U)\geq 4R(V)/(4R(V)+1)$ then 
	$$	||g||^2_{L^2(\mathbb{T})}\leq \frac{2}{\rho(U)}\int_{U} |g(x)|^2dx. $$
\end{claim}
\begin{proof}
	First, we may  write 
	$$g(x)= \sum_{i} \widehat{g}(n_i)e(n_i \cdot x).$$
Thus, separating the diagonal terms from the others, we have
	\begin{align}
		\label{b.5}	\int_U|g(x)|^2dx&= \rho(U)\sum_i |\widehat{g}(n_i)|^2 + \sum_{i\neq j}\widehat{ \mathds{1}}_{U}(n_i-n_j)\widehat{g}(n_i)\overline{\widehat{g}(n_j)} \nonumber \\
		&= \rho(U)||g||^2_{L^2(\mathbb{T})} + \langle Q_U g,g\rangle,
	\end{align}
	where $Q_U=(q_{ij})$ is an operator on $L^2(\mathbb{T})$ with matrix representation, in the base $\{e(nx)\}_{n\in \mathbb{Z}}$, given by 
	\begin{align}
		\label{operator}	q_{ij}= \begin{cases}
			\widehat{ \mathds{1}}_{U}(n_i-n_j) & n_i\neq n_j \\
			0 & \text{otherwise}
		\end{cases}.
	\end{align} 
	Since $\mathds{1}_{U}(\cdot)$ is real-valued, 	$\widehat{ \mathds{1}}_{U}(-n)= 	\overline{\widehat{ \mathds{1}}_{U}}(n)$, thus $Q_U$ is a self-adjoint operator whose Hilbert-Schmidt norm is bounded by
	\begin{align}
		\label{B.4}
		||Q_U|| \leq R(V)^{1/2} \left(\sum_{n\neq 0} \left| \widehat{ \mathds{1}}_{U}(n)\right|^2\right)^{1/2}= (R(V)\rho(U)(1-\rho(U)))^{1/2}.
	\end{align}
	In particular, if $\rho(U)\geq 4R(V)/(4R(V)+1)$, we have $(R(V)\rho(U)(1-\rho(U)))^{1/2}\leq \rho(U)/2$, thus \eqref{B.4} together with \eqref{b.5} give Claim \ref{trivial range}. 
\end{proof}
We are finally ready to prove Theorem \ref{Thm: Logintegrability}:
\begin{proof}[Proof of Theorem \ref{Thm: Logintegrability}]
	Let $0<\nu\leq 4R(V)/(4R(V)+1)$ be some parameter and denote by $A(\nu)$ the smallest constant such that 
	$$ ||g||^2_{L^2(\mathbb{T})}\leq A(\nu)\int_{U} |g(x)|^2dx,$$
	uniformly for all $g\in L^2(\mathbb{T})$ with $\Spec(g)\subset V$ and any set $U\subset \mathbb{T}$ with $\rho(U)\geq \nu$. Moreover, let $\varphi(\nu)= \log A(\nu)$, $\Delta(\nu) = \nu^{1+\frac{2}{p}}(4G)^{-1}$ with  $G$ given by Claim \ref{choice of m}. Applying Lemma \ref{Spreading lemma}, bearing in mind that $m\leq  G\nu^{-\frac{2}{p}}$, we obtain a set $U_1\subset \mathbb{T}$ of measure $\rho(U_1)\geq \nu +\Delta(\nu)$ such that 
	$$		\int_{U_1} |g(x)|^2 dx \leq \left(\frac{C G^5}{\nu^{2+\frac{10}{p}}}\right)^{3G \nu^{-\frac{2}{p}}}\int_{U} |g(x)|^2 dx.$$
	for some constant $C>0$. Since, by definition of $A(\cdot)$,
	$$ ||g||^2_{L^2(\mathbb{T}^2)}\leq A(\nu +\Delta(\nu)) \int_{U_1} |g(x)|^2 dx,$$
	we have 
	$$ A(\nu) \leq A(\nu +\Delta(\nu)) \left(\frac{C G^5}{\nu^{2+\frac{10}{p}}}\right)^{3G \nu^{-\frac{2}{p}}},$$
	and taking the logarithm of both sides, we finally deduce
	\begin{align}
		\label{dif equation}\frac{\varphi(\nu)- \varphi(\nu +\Delta(\nu))}{\Delta(\nu)}\leq \frac{12G^2}{\nu^{1+\frac{4}{p}}}\log \frac{C G^5}{\nu^{2+\frac{10}{p}}} \leq \frac{C_1(\varepsilon) G^3}{\nu^{1+\frac{4}{p}+\varepsilon}},
	\end{align}
for some constant $C_1(\varepsilon)>0$.	Comparing \eqref{dif equation} with the differential inequality $d\varphi(\nu)/d\nu\leq C(\varepsilon) G\nu^{-1-\frac{4}{p}-\varepsilon}$, in light of the fact that $A(\nu)$ is increasing, we deduce that 
	$$ \varphi(\nu)\leq C(\varepsilon)C_0^6 R(V)^3 \nu^{-1-\frac{4}{p}-\varepsilon},$$
	where we have used the definition of $G$ given by Claim \ref{choice of m}. If $\nu(U)\geq 4R(V)/(4R(V)+1)$, then Claim \ref{trivial range} shows that the conclusion of Theorem \ref{Thm: Logintegrability} is still satisfied. 
\end{proof}
\subsection{Proof of Lemma \ref{Spreading lemma}.}
\label{proof of Lemma spreading}
In this section we prove Lemma \ref{Spreading lemma}. The proof follows closely the arguments in \cite[Section 3.4]{N93}, again we claim no originality. We will need the following definition: 
\begin{defn}
	Let $m$ be a positive integer and let $\tau,\varkappa>0$ be some parameters. Given $g\in L^2(\mathbb{T})$, we say that $g\in EP_{\text{loc}}^m(\tau,\varkappa)$ if for every $t\in (0,\tau)$ there exist constants $a_0(t),...,a_m(t)\in \mathbb{C}$ such that 
	$\sum_k |a_k|^2=1$ and 
	$$ \left|\left|\sum_{k=0}^m a_k(t)g_{kt}\right|\right|_{L^2(\mathbb{T})}\leq \varkappa,$$
	where $g_{kt}(\cdot):= g(\cdot+kt).$
\end{defn}
We refer the reader to \cite[Section 3.1-3.4]{N93} for an accurate description of the class $EP_{\text{loc}}^m(\tau,\varkappa)$. Intuitively, functions in  $EP_{\text{loc}}^m(\tau,\varkappa)$ \textquotedblleft behave like\textquotedblright trigonometric polynomials of degree $m$ in intervals of length $\tau$ up to an error $\varkappa$. The key estimate that we will need is the following \cite[Corollary 3.5']{N93}: 
\begin{lem}
	\label{main prop EPloc}
	Let $g\in EP_{\text{loc}}^m(\tau,\varkappa)$ for some integer $m>0$ and some $\tau,\varkappa>0$. Moreover, let $U\subset \mathbb{T}\subset \mathbb{R}^2$ be a set of positive measure and $\nu:= \rho(e(m\tau)U \backslash U)$.  There exists a set $U_1\supset U$ of measure $\rho(U_1\backslash U) \geq \frac{\nu}{2}$ such that  
	$$\int_{U_1} |g(x)|^2dx\leq \left(\frac{C m^3}{\nu^2}\right)^{2m} \left(\int_U |g(x)|^2dx + \varkappa^2\right),$$
	for some constant $C>0$. 
\end{lem}
We will also need the following two claims: 
\begin{claim}
	\label{subspace claim}
	Let $U\subset \mathbb{T}$ be a measurable subset and let $m$ be as in Lemma \ref{Spreading lemma}. Then, there exists a subspace $V_{m}$ of $L^2(\mathbb{T})$ of dimension at most $m$ such that for all $g\in L^2(\mathbb{T})$ orthogonal to $
	V_m$, we have 
	$$ ||g||^2_{L^2(\mathbb{T})}\leq \frac{2}{\rho(U')} \int_{U'} |g(x)|^2dx $$
	for all subsets $U'\subset U$ with $\rho(U')\geq \rho(U)/2$
\end{claim}
\begin{proof}
	Indeed, let $|\sigma_1|\leq|\sigma_2|\leq ...$ be the eigenvalues of the operator $Q_{U'}$ defined in \eqref{operator} with $U'$ instead of $U$. Then we take $V_m$ to be the subspace generated by the eigenvectors with eigenvalues $\sigma_1,...,\sigma_m$. We are now going to show that $V_m$ has the claimed property.  By definition of $m$, we have  
	$$ \sum_i |\sigma_i|^2= ||Q_{U'}||^2 \leq  \sum_{i\neq j}|q_{ij}|^2 = \sum_{n_i\neq n_j} \left| \widehat{ \mathds{1}}_{U'}(n_i-n_j)\right|^2\leq \frac{\rho(U')^2(m+1)}{4}.$$
	Thus,
	$$ |\sigma_{m+1}|^2\leq \frac{1}{m+1}\cdot \frac{\rho(U')^2(m+1)}{4}\leq \frac{\rho(U')^2}{4}.$$
	Therefore Claim \ref{subspace claim} follows from that fact that the norm of $Q_{U'}$ restricted to   $L^2(\mathbb{T})\backslash V_m$ is at most $ |\sigma_{m+1}|\leq \rho(U')/2$ and an analogous argument to Claim \ref{trivial range}.  
\end{proof}
Now, if $V$ is finite we let $N= |V|$, if $V$ is infinite we can ignore the dependence on $N$ in the rest of the argument. With this notation, we claim the following:
\begin{claim}
	\label{local behaviour of g}
	Let $U\subset \mathbb{T}$ be a measurable set, $m$ be as in Lemma \ref{Spreading lemma} and, if $V$ is finite, suppose that $m<N$, moreover let $g\in L^2(\mathbb{T})$ with $\Spec(g)\subset V$. Then there exists some $\sigma \in (0,1)$ such that $g\in EP^m_{\text{loc}}(\tau, \varkappa)$ where $\varkappa^2=\frac{4}{\rho(U)} (m+1)\int_U |g(x)|^2dx$, $\tau= \sigma/2m$ and, moreover $\nu:= \rho (e(m\tau)U \backslash U)\geq \rho(U)/2m$. 
\end{claim}
\begin{proof}
	Let $t\in [0,1)$ be given, since  exponentials with different frequencies are linearly independent\footnote{Suppose that $n_i\neq n_j$ for $i\neq j$ and $\sum_i a_{n_i}e(n_ix)=0$. Multiplying both sides by $e(-n_1x)$ and integrating for $x\in \mathbb{T}$, we see that $a_1=0$. Repeating the argument, we get $a_i=0$ for all $i$.} in $L^2(\mathbb{T})$, we can choose coefficients $a_k(t)$, so that $\sum_k |a_k|^2=1$ and the function
	$$h(\cdot)= \sum_{k=0}^m a_k(t)g_{kt}(\cdot),$$
	where $g_{kt}(x)=\sum a_{n_i}e(n_ikt)e(n_ix)$, is orthogonal to $V_m$, given in Claim \ref{subspace claim}, provided that $m< N$. Therefore, Claim \ref{subspace claim} gives 
	\begin{align}
		||h||^2_{L^2(\mathbb{T})} \leq \frac{2}{\rho(U')} \int_{U'} |h(x)|^2dx, \label{B.7}
	\end{align}
	for all $U'\subset U$ with $\rho(U')\geq \rho(U)/2$. 
	
	We are now going to choose an appropriate set $U'$ in order to estimate the r.h.s. of \eqref{B.7}. Let $t\geq 0$ and take  $U'=U_t:= \cap_{k=0}^m e(-kt)U$, 	since the function $t\rightarrow\rho( U_t \backslash U)$ is continuous and takes value $0$ at $t=0$, we can find some sufficiently small $\tau>0$ so that, for all $t\in (0,\tau)$, the set $U_t:= \cap_{k=0}^m e(-kt)U$ has measure at least $\rho(U)/2$. To estimate the r.h.s. of \eqref{B.7}, we observe that, for every $k=0,...,m$, we have  
	$$\int_{U_{t}}|g_{kt}(x)|^2dx\leq \int_{e(-kt)U} |g_{kt}(x)|^2dx = \int_U |g(x)|^2dx.$$
	Thus, the Cauchy-Schwarz inequality gives
	\begin{align} 
		\label{B.6}
		\int_{U_t}|h(x)|^2dx \leq \left(\sum_{k=0}^m  \int_{U_t}|g_{kt}(x)|^2dx \right)\leq (m+1)\int_U |g(x)|^2dx.
	\end{align}
	Hence, \eqref{B.7} together with \eqref{B.6}, bearing in mind that $\rho(U_t)\geq \rho(U)/2$, give that for all $t\in (0,\tau)$ there exists coefficients $a_1(t),...,a_m(t)$ such that $\sum_k |a_k|^2=1$ and  
	$$ \left|\left|\sum_{k=0}^m a_k(t)g_{kt}\right|\right|^2_{L^2(\mathbb{T})}\leq \frac{4(m+1)}{\rho(U)}\int_U |g(x)|^2dx.$$

	We are now left with proving the claimed estimates on $\tau$ and $\nu$. Let $\psi(s)= \rho(e(s)U \backslash U)$, bearing in mind that $\rho(U)\leq 4R(V)/(4R(V)+1)$ so that, by \eqref{B.2},  $\rho(\mathbb{T}\backslash U)\geq (4R(V)+1)^{-1}\geq (2m)^{-1}$, we have  
	$$ \int_0^1\psi(s)ds= \rho(U)\rho(\mathbb{T}\backslash U) \geq \frac{\rho(U)}{2m}.$$
	Thus, since $\psi(s)$ is non-negative and continuous, there exists some $\sigma\in (0,1)$ such that for all $s\leq \sigma$ we have $ \rho (e(s)U \backslash U) \leq \rho(U)/2m$. 
	We now verify that such $\tau=\sigma/m$ satisfies $\rho(U_t)\geq \rho(U)/2$ for all $t\in (0,\tau)$. Indeed, bearing in mind that $kt\in (0,m\tau)$, we have 
	\begin{align}
		\rho(U_t)= \rho \left( \cap_{k=0}^m e(-kt)U\right) \geq \rho(U)- \sum_{k=1}^m \rho( e(kt)U\backslash U) \geq \rho(U)- m\frac{\rho(U)}{2m}\geq \rho(U)/2,
	\end{align}
	concluding the proof of Claim \ref{local behaviour of g}. 
\end{proof} 
We are finally ready to present the proof of Lemma \ref{Spreading lemma}: 
\begin{proof}[Proof of Lemma \ref{Spreading lemma}]

	Suppose that $m< N$, then, applying Lemma \ref{main prop EPloc} with the choice of parameters given by Claim \ref{local behaviour of g}, we obtain part (1) of Lemma \ref{Spreading lemma}. For part (2), Lemma \ref{main prop EPloc} gives
	\begin{align}
		\int_{U_1}|g(x)|^2dx &\leq \left(\frac{Cm^5}{\rho(U)^2}\right)^{2m} \left(\frac{4(m+1)}{\rho(U)}+1\right)\int_U |g(x)|^2dx \nonumber \\
		& \leq \left(\frac{Cm^5}{\rho(U)^2}\right)^{3m}\int_U |g(x)|^2dx, \label{B.10}
	\end{align}
	as required. 
	
	Let us now suppose that $m\geq N$, then the Nazarov-Tur\'{a}n Lemma \cite[Theorem 1]{N93}, for any set $U_1\subset \mathbb{T}$ of measure $\rho(U_1)= \rho(U) + \rho(U)/4m$, gives  
	\begin{align}
		\int_{U_1}|g(x)|^2dx &\leq  \left(\frac{C\rho(U_1)}{\rho(U)}\right)^{N-1} \int_U |g(x)|^2dx \nonumber \\
		&\leq \left(C+ \frac{C}{4m}\right)^{N-1} \int_U |g(x)|^2dx \nonumber \\
			&\leq 100C^{m} \int_U |g(x)|^2dx \nonumber
	\end{align}
	and \eqref{B.10} follows, up noticing that $\rho(U)\leq 1$. 
\end{proof}

		\bibliographystyle{siam}
	\bibliography{Logtoral}
\end{document}